\documentclass[12pt,english]{article}
\usepackage{babel}
\usepackage[cp1250]{inputenc}
\usepackage[T1]{fontenc}
\usepackage{amsfonts}
\usepackage{amsmath}
\usepackage{amsthm}
\usepackage{mathrsfs}
\usepackage{hyperref}
\usepackage{enumerate}
\usepackage{graphicx}
\usepackage{pictex}
\usepackage[dvipsnames]{xcolor}

\newcommand{\eqq}[2]{\begin{equation}  #1  \label{#2}
\end{equation}    }
\newcommand*{\norm}[1]{\left\Vert{#1}\right\Vert}

\def\divv{\operatorname{div}}
\newcommand{\hd}{\hspace{0.2cm}}
\newcommand{\no}{\noindent}
\newcommand{\m}[1]{\mbox{#1}}
\newcommand{\rr}{\mathbb{R}}

\newcommand{\lap}{\Delta}
\newcommand{\lapk}{\Delta^{2}}

\newcommand{\nal}{\nabla \Delta}

\newcommand{\nab}{\nabla}
\newcommand{\nabd}{\nabla^{2}}
\newcommand{\nabt}{\nabla^{3}}

\newcommand{\re}[1]{\frac{1}{#1}}

\newcommand{\jd}{\frac{1}{2}}
\newcommand{\td}{\frac{3}{2}}
\newcommand{\jc}{\frac{1}{4}}
\newcommand{\tc}{\frac{3}{4}}
\newcommand{\ddt}{\frac{d}{dt}}

\newcommand{\bk}[1]{\left|  #1 \right|^{2}}

\newcommand{\nd}[1]{\left\| #1  \right\|_{2}}

\newcommand{\ns}[1]{\| #1  \|_{6}}
\newcommand{\nsk}[1]{\| #1  \|_{6}^{2}}

\newcommand{\nt}[1]{\| #1  \|_{3}}

\newcommand{\nc}[1]{\| #1  \|_{4}}
\newcommand{\nck}[1]{\| #1  \|_{4}^{2}}

\newcommand{\nj}[1]{\| #1  \|_{1}}
\newcommand{\ndk}[1]{\| #1  \|_{2}^{2}}
\newcommand{\nif}[1]{\left\| #1  \right\|_{\infty}}
\newcommand{\nifk}[1]{\left\| #1  \right\|_{\infty}^{2}}

\newcommand{\nsod}[1]{\| #1 \|_{2,2} }

\newcommand{\nsotk}[1]{\| #1 \|_{3,2}^{2} }

\newcommand{\nsodk}[1]{\| #1 \|_{2,2}^{2} }

\newtheorem{rem}{{\textbf {Remark}}}

\newtheorem{prop}{{\textbf {Proposition}}}
\newtheorem{theorem}{\textbf {Theorem}}
\newtheorem{coro}{\textbf  {Corollary} }

\newcommand{\vt}{v_{,t}}
\newcommand{\ot}{\omega_{,t}}
\newcommand{\bt}{b_{,t}}
\newcommand{\om}{\omega}
\newcommand{\Om}{\Omega}
\newcommand{\OmT}{\Om^{T}}
\newcommand{\bno}{\frac{b}{\om}}
\newcommand{\dvk}{|D(v)|^{2}}

\newcommand{\bmi}{b_{\min}}
\newcommand{\bma}{b_{\max}}
\newcommand{\omi}{\om_{\min}}
\newcommand{\oma}{\om_{\max}}
\newcommand{\bmt}{b_{\min}^{t}}

\newcommand{\omt}{\om_{\min}^{t}}
\newcommand{\ommt}{\om_{\max}^{t}}
\newcommand{\bmmt}{b_{\max}(t)}

\newcommand{\mnit}{\mu^{t}_{\min}}

\newcommand{\kd}{\kappa_{2}}

\newcommand{\bmit}{\frac{\bmi}{\left( 1 + \kd \oma t \right)^{\frac{1}{\kd}}}}

\newcommand{\omitt}{\frac{\omi}{1 + \kd \omi t}}
\newcommand{\omittau}{\frac{\omi}{1 + \kd \omi \tau}}
\newcommand{\omits}{\frac{\omi}{1 + \kd \omi s}}
\newcommand{\omat}{\frac{\oma}{1 + \kd \oma t}}

\newcommand{\V}{\mathcal{V}}

\newcommand{\Vt}{\V^{3}}
\newcommand{\Vj}{\V^{1}}

\newcommand{\Vkd}{\dot{\V}_{\divv}}
\newcommand{\Vkdj}{\dot{\V}_{\divv}^{1}}

\newcommand{\Vkdt}{\Vkd^{3}}

\newcommand{\n}[1]{ \left( #1  \right)}

\newcommand{\nabk}{\nabla^{2}}

\newcommand{\ts}{t^{*}}
\newcommand{\tjs}{t_{1}^{*}}
\newcommand{\tskd}{\ts_{K,\delta}}

\newcommand{\wt}[1]{\widetilde{ #1}}
\newcommand{\comkd}{C_{\Om,\kd}}

\newcommand{\jkt}{(1+\kd \oma t )}

\newcommand{\Ts}{T^{*}}

\newcommand{\iOm}{\int_{\Om}}

\newcommand{\ommm}{(\om - \omt)_{-}}
\newcommand{\ommp}{(\om - \ommt)_{+}}

\newcommand{\bbtmm}{(b-\bmt)_{-}}
\newcommand{\bom}{\frac{b}{\om}}

\newcommand{\vj}{v^{1}}
\newcommand{\bj}{b^{1}}
\newcommand{\oj}{\om^{1}}
\newcommand{\vd}{v^{2}}
\newcommand{\bd}{b^{2}}
\newcommand{\od}{\om^{2}}
\newcommand{\bjoj}{\frac{\bj}{\oj}}
\newcommand{\bdod}{\frac{\bd}{\od}}
\newcommand{\boj}{\frac{b}{\oj}}
\newcommand{\joj}{\frac{1}{\oj}}
\newcommand{\jojod}{\frac{1}{\oj\od}}

\newcommand{\dl}{\delta}
\newcommand{\dlj}{\dl_{1}}
\newcommand{\dld}{\dl_{2}}
\newcommand{\dlt}{\dl_{3}}

\newcommand{\lapy}{\| \lap v_{0}\|_{2}^{2}+\| \lap \om_{0}\|_{2}^{2}+\| \lap b_{0}\|_{2}^{2}}

\newcommand{\eqns}[1]{
\begin{eqnarray*}
\begin{split}
#1
\end{split}
\end{eqnarray*}
}
\newcommand{\eqnsl}[2]{
\begin{equation}
\label{#2}
\begin{split}
#1
\end{split}
\end{equation}
}
\newcommand{\eqn}[1]{
\begin{eqnarray*}
\begin{split}
#1
\end{split}
\end{eqnarray*}
}

\begin{document}
\title{Global in time solution to Kolmogorov's two-equation model of turbulence with small initial data}

\author{Przemys\l aw Kosewski, Adam Kubica \footnote{Department of Mathematics and Information Sciences, Warsaw University of Technology, ul. Koszykowa 75,
00-662 Warsaw, Poland, E-mail addresses: A.Kubica@mini.pw.edu.pl, P.Kosewski@mini.pw.edu.pl} }

\maketitle

\abstract{We prove the existence of global in time solution to Kolmogorov's two-equation model of turbulence in three dimensional domain with periodic boundary conditions under smallness assumption imposed on initial data.   }

\vspace{0.3cm}

\no Keywords: Kolmogorov's two-equation model of turbulence, global in time solution, small data.

\vspace{0.2cm}


\section{Introduction}

An analysis of turbulent motion is one of the most challenging scientific problems. There are many models (e.g. $k - \varepsilon$, $k - \omega$, see \cite{Lew}) that give us some insight onto this phenomenon, but our understanding of it is still insufficient. One of the models, proposed by A. N. Kolmogorov in 1941, is the two-equation model of turbulence. To the best of our knowledge, only a few research papers are devoted to the mathematical analysis of this problem.  Firstly, we would like to recall the  Kolmogorov's two equation turbulence model
\eqq{\vt + \divv(v\otimes v) -  2\nu_{0} \divv \n{ \bno  D(v)} = - \nabla p  , }{a}
\eqq{ \ot + \divv(\om v ) - \kappa_{1} \divv \n{ \bno \nabla \om } = - \kappa_{2} \om^{2},   }{b}
\eqq{\bt + \divv(b v   )   - \kappa_{3} \divv \n{  \bno \nabla b }  = - b \om + \kappa_{4} \bno \dvk,   }{c}
\eqq{\divv{v} =0,}{d}
where $v$ - mean velocity, $\omega$ - dissipation rate, $b$ - 2/3 of mean kinetic energy, $p$ - sum of
mean pressure and $b$.
Despite of its huge importance, it still remains relatively little-studied. For a more exhaustive introduction to Kolmogorov's two equation turbulence model we refer to \cite{BuM}, \cite{Lew}, \cite{Kolmog}, \cite{KoKu}, \cite{MiNa}, \cite{MiNaa}, \cite{Spal}.

Now, we will shortly describe the known results concerning this model. In \cite{BuM}, the authors consider the system in a bounded  $C^{1,1}$ domain with mixed boundary conditions for $b$ and $\om$ and stick-slip boundary condition for velocity $v$. In order to overcome the difficulties related with the last term on the right hand side of (\ref{c}) the problem is reformulated and the quantity $E:=\frac{1}{2}|v|^{2}+ \frac{2\nu_{0}}{\kappa_{4}}b$ is introduced. Then, the equation (\ref{c}) is replaced by
\[
E_{,t}+ \divv(v(E+p))- 2\nu_{0}\divv\left( \frac{\kappa_{3} b}{\kappa_{4}\om}\nab b+ \frac{b}{\om} D(v)v \right)+\frac{2\nu_{0}}{\kappa_{4}}b\om=0.
\]
Then, there is established the existence of global-in-time weak solution to the reformulated problem. It is also worth mentioning that in \cite{BuM} the assumption related to the initial value of $b$ admit vanishing of $b_{0}$ in some points of the domain. More precisely, the existence of weak solution is proved under the conditions $b_{0}\in L^{1}$, $b_{0}>0$ a.e. and $\ln{b_{0}}\in L^{1}$.

In the paper \cite{MiNa}, the system (\ref{a})-(\ref{d})  is considered in periodic domain. It is proved the existence of global-in-time weak solution, but due to the presence of strongly nonlinear term $\bno \dvk$,  the weak form of equation (\ref{c}) have to be corrected by a positive measure $\mu$, which is zero, provided weak solution is sufficiently regular. There are also obtained the estimates for $\om$ and $b$ (see (4.2) in \cite{MiNa}). These observations are crucial in our reasoning presented below. Concerning the initial value of $b$, the authors assume that $b_{0}$ is uniformly positive.

In \cite{KoKu} local-in-time existence of solution to the system (\ref{a})-(\ref{d}) with periodic boundary condition is studied. More precisely, if the initial data belongs to Sobolev space $H^2(\Omega)$ and  $b_0(x) \geq \bmi > 0 $, $\om_0(x) \geq \omi > 0 $, then there exists a "regular" solution defined on some interval  $[0,t^*)$. Furthermore, it is showed that the  solution belongs to $L^2(0,t^*;H^3(\Omega))\cap H^1(0,t^*;H^1(\Omega))\cap L^\infty(0,t^*;H^2(\Omega))$. Additionally, an estimate for minimal time of existence of solution in terms of initial data is proven. This last result is crucial in our proof of the existence of global-in-time solution.

It is worth to mention the other publications regarding mathematical analysis of turbulence models e.g: in \cite{Pares} the author analyses 0-equation model of turbulence (the turbulent viscosity is related with  the symmetric gradient of velocity only).  In \cite{NaumannV2} it is analysed a simplified 1-equation model  of turbulence (Prandtl's model, see \cite{Prandtl}). In the paper \cite{HB} a  stationary 1-equation model of turbulence in porous medium is studied. 
The paper \cite{Dreyfus} is devoted to  a simplified scalar version of the RANS model arising in oceanography.
Very recently in \cite{Fanelli} the authors studied a system very closely related to one-dimensional Kolmogorov system. Local well-posedness was shown even with vanishing mean turbulent kinetic energy. It was also proved that for some smooth initial data the obtained solutions blow-up in finite time.

In the presented paper we formulate the smallness conditions, which guarantee the global-in-time existence of solution to (\ref{a})-(\ref{d}). These results are given in Theorem~\ref{TW_GLOBAL} and Corollary~\ref{coro_glob}.

At the outset  we will establish  the  basic notation. Assume that $\Omega = \prod_{i=1}^{3}(0,L_{i}) $, \hd $L_{i}>0$ and $\Omega^{T}=\Omega \times (0,T)$.  We  shall consider  problem (\ref{a})-(\ref{d}) in $\Om^{T}$, where
\eqq{v,\om ,b \hd \m{ are periodic on }  \hd \Om, \hd \hd \int_{\Om} vdx =0,}{ddod}
with initial condition
\eqq{v_{|t=0}= v_{0}, \hd \hd \om_{|t=0}= \om_{0}, \hd \hd b_{|t=0}= b_{0}.}{e}
Here $\nu_{0}, \kappa_{1}, \dots, \kappa_{4}$ are positive constants. We assume that all these constants except $\kd$ are equal to one. As we will see,  $ \kd $ plays a special role in our system  and  it determines the long-time behaviour of the fraction $ \frac{b}{\om}$.
We additionally assume that there exist positive numbers $\bmi$, $\omi$, $\oma$ such that
\eqq{0<\bmi\leq b_{0}(x),  }{f}
\eqq{0<\omi\leq \om_{0}(x) \leq \oma  }{g}
on $\Omega$. In the next section we will introduce notation dedicated to formulate smallness conditions as well as auxiliary theorem that will be useful in further part of work.
\section{Notation and auxiliary theorem}
In this section we introduce the notation. We will use the standard notation
\[
\| f \|_{p}= \left( \int_{\Om} |f(x)|^{p} dx \right)^{\frac{1}{p}}
\]
and  we set
\eqq{
\omt = \omitt, \hd \hd
\ommt = \omat. \hd \hd
}{omMinMaxt}
These quantities will appear in the  lower and upper  bound  for $\om$  (see Proposition~\ref{oszac_og}). Additionally, we introduce the analogous notation for $b$, for the lower bound for $b$ and the upper bound for $\| b\|_{1}$ (see Proposition~\ref{oszac_og} and Proposition~\ref{propEstimates}c)
\eqnsl{
\bmt  = \bmit,  \hd
\bmmt  =
\frac{\nj{b_0} + \jd \ndk{v_0}\n{
1 + I_\infty \n{\kd, \frac{\omi}{\oma}, \frac{\bmi}{\n{\oma}^2}}
}}{\n{1 + \kd \omi t }^\frac{1}{\kd}},
}{bMinMaxt}
where
\eqnsl{
I_\infty  \n{\kd, x, y}
 = \Gamma\n{\frac{2\kd }{2\kd - 1}}
x^{\min \left\{ 1, \frac{1}{\kd} \right\}}
\n{\frac{C_p^2  (2\kd -1)}{2y}
\exp \n{\frac{2y}{C^{2}_{p} }   }
}^{\frac{1}{2\kd -1}},
}{IinfDef}
and $C_{p}$ is Poincar\'e constant for the domain $\Om$, i.e. the smallest constant such that  $\|f\|_{p}\leq C_{p}\| \nab f\|_{p} $ for  smooth $f$ such that $\int_{\Om} fdx =0$.
In case of $b$ we will  be able to control the decay of $L^1$-norm. Frequently, we will estimate from below the coefficient in the diffusive term by (see (\ref{below}))
\eqq{
\mnit = \frac{\bmt}{\ommt}= \frac{\bmi}{\oma}\jkt^{1-\frac{1}{\kd}}.
}{mnitDef}
To express the smallness of initial data we will need the following quantity
\eqnsl{
Y_{2}(t)=  \Big(\ndk{\lap b_0} + &  \ndk{\lap \om_0 } + \ndk{\lap v_0 } \Big) \cdot  \\
 \cdot  &  \exp \n{ - \frac{1}{C_{p}^{2}} \frac{\bmi }{(2\kd - 1)\oma^2}  \n{ \n{1 + \kd \oma t}^{2 - 1/\kd} - 1}}.\\
}{Y_0t}
Furthermore, to formulate a condition that ensures the existence of global-in-time solution we have to define  (see (\ref{GLOB_ADD}) in Theorem~\ref{TW_GLOBAL})
\eqq{
A(t) =
\n{\ndk{v_0}
\exp \n{ -\frac{
2\bmi \n{\n{1 + \kd \oma t }^{2 - \frac{1}{\kd}}-1} }{C^{2}_{p} \oma^{2} \n{2 \kd - 1 }}   }
+ \bmmt^2}^\jc,
}{def_A}
\eqq{
B(t) =
1 + \re{\omt} + \frac{\bmmt}{\omt} + \frac{\bmmt}{(\omt)^2} ,
}{def_B}
\eqq{
C(t) =
\re{\omt} + \re{(\omt)^2} + \frac{\bmmt}{(\omt)^2} + \frac{\bmmt}{(\omt)^3},
}{def_C}
\eqq{
D(t) = \re{(\omt)^2} + \re{(\omt)^3}
}{def_D}
and
\eqq{
Z_0(t) =
\Big( \bmmt +
A(t) Y_2^\jc(t)
+ B(t) Y_2^\jd(t)
+ C(t) Y_2 (t)
+ D(t) Y_2^\td(t)
\Big).
}{def_Z0}
Now, let us  define function spaces. If $m\in \mathbb{N}$, then  we denote by $\V^{m}$ the space of restrictions to $\Om$ of the functions, which belong to the space
\eqq{\{u \in H^{m}_{loc}(\rr^{3}): \hd  u(\cdot + kL_{i}e_{i}) = u(\cdot ) \hd \m{ for } \hd k\in \mathbb{Z}, \hd  i=1,2,3 \}.}{j}
Next, we set
\eqq{\Vkd^{m} = \{v\in \V^{m}:\hd \divv v =0, \hd \int_{\Om} vdx =0  \}.}{k}
We shall find global solution of the system (\ref{a})-(\ref{e}) such that $(v,\om,b)\in \mathcal{X}(T)$, where
\eqq{
\mathcal{X}(T)= L^{2}_{loc}([0,T);\Vkdt)\times (L^{2}_{loc}([0,T);\Vt))^2) \cap (H^{1}_{loc}([0,T);H^{1}(\Om)))^{5}.
}{i}
We denote by $\| \cdot \|_{k,2}$ the norm in the  Sobolev space, i.e.
\[
\| f \|_{k,2} = (\ndk{\nab^{k} f } + \ndk{f})^{\frac{1}{2}},
\]
where $\| \cdot \|_{2}$ is $L^{2}$ norm on $\Om$.

Now, we introduce the notion of solution to the system (\ref{a})-(\ref{d}). We shall show that for any $v_{0}\in \Vkd^{2}$ and strictly positive $\om_{0}$, $b_{0}\in \V^{2}$, if $H^{2}$ norms of the initial data are sufficiently small, then there exist  $(v, \om, b)\in \mathcal{X}(\infty)$ such that
\eqq{(\vt, w) - (v\otimes v,\nabla w) +    \n{ \mu  D(v), D(w)} = 0  \hd \m{ for } \hd w\in \Vkd^{1}, }{aa}
\eqq{ (\ot, z) - (\om v, \nabla z  ) + \n{ \mu \nabla \om, \nabla z  } = - \kd (\om^{2}, z ) \hd \m{ for } \hd z\in \V^{1},    }{ab}
\eqq{(\bt, q) - (b v, \nab q   )   +\n{  \mu  \nabla b, \nabla q  }  = - (b \om , q) + ( \mu \dvk, q) \hd \m{ for } \hd q\in \V^{1},  }{ac}
for a.a. $t\in (0,T)$, where $\mu= \frac{b}{\om}$ and (\ref{e}) holds. Recall that  $D(v)$ denotes the symmetric part of $\nabla v$ and $(\cdot , \cdot )$ is inner product in $L^{2}(\Om)$.

In \cite{KoKu} it was shown that for appropriately regular initial data there exists local-in-time regular solution. We recall this result below.
\begin{theorem}[theorem 1 \cite{KoKu}]
Suppose  that  $\om_{0}$, $b_{0}\in \V^{2}$,  $v_{0}\in \Vkd^{2}$ and (\ref{f}), (\ref{g}) are satisfied. Then there exist positive $\ts$ and $(v,\om , b)\in \mathcal{X}(\ts)$ such that  (\ref{e}), (\ref{aa})-(\ref{ac}) holds for a.a. $t\in (0,\ts)$.
Furthermore, for each $(x,t)\in \Om\times [0,\ts)$ the following estimates
\eqq{\frac{\omi}{1+\kd \omi t} \leq \om(x,t) \leq \frac{\oma}{1+\kd \oma t},}{newc}
\eqq{\frac{\bmi}{(1+\kd \oma t)^{\frac{1}{\kd}}}  \leq b(x,t) }{newd}
hold. The time of existence of solution is estimated from below in the following sense: for each positive $\delta$ and compact $K\subseteq \{ (a,b,c): 0<a\leq b, \hd 0<c \}$ there exists positive $\tskd$, which depends only on $\kd, \Om, \delta$ and $K$ such that if
\eqq{  \nsodk{v_{0}}+\nsodk{\om_{0}}+\nsodk{b_{0}}\leq \delta  \hd \m{ and } \hd (\omi, \oma, \bmi)\in K,}{uniformly}
then $\ts\geq \tskd$.
\label{LOCALNE_TH}
\end{theorem}

\section{Main result}
Now, we  formulate the main result  involving the global existence of regular solutions to system (\ref{a})-(\ref{e}).
\begin{theorem}\label{TW_GLOBAL}
Assume that  $\kd > \frac{1}{2}$. There exists a constant $C_{\Om,\kd}$, which depends only on $\Om$ and $\kd$, with the following property: for any
  $\om_{0}$, $b_{0}\in \V^{2}$,  $v_{0}\in \Vkd^{2}$, if   (\ref{f}), (\ref{g}) hold and
\eqq{
\mnit -C_{\Om, \kd} Z_0(t) > 0 ~~~~\m{ for } \hd  t \in [0, T),
}{GLOB_ADD}
for some $T\in (0,\infty]$, then there exists a unique  $(v,\om , b)\in \mathcal{X}(T)$ solution to (\ref{a})-(\ref{e}) in $\OmT$.
\end{theorem}
We recall that we impose the constants $\nu_{0}$, $\kappa_{1}$, $\kappa_{3}$ and $\kappa_{4}$ are equal to one. In general case, if all these constants are positive and arbitrary, then the constant in the above result will depend on $\nu_{0}$, $\kappa_{1}$, \dots, $\kappa_{4}$ and $\Om$. The functions $\mnit$ and   $Z_{0}(t)$ were defined in (\ref{mnitDef}) and (\ref{def_Z0}), respectively.

\begin{rem}
The condition (\ref{GLOB_ADD}) involves only the initial data: $v_{0}$, $\om_{0}$, $b_{0}$,  the parameters of the system: $\nu_{0}$, $\kappa_{1}$, \dots , $\kappa_{4}$ and $\Om$.
\end{rem}

\begin{rem}
The assumption $\kd > \jd$ is crucial in the proof of   Theorem \ref{TW_GLOBAL} (and also in Proposition \ref{propEstimates}). Without it we are unable to prove the exponential decay of $L^2$-norm of $v(t)$ and polynomial decay of $L^1$-norm of $b(t)$, around which the proof is structured. 
\end{rem}

\begin{rem}
As is stated in \cite{Spal}, Kolmogorov set $\kd = \frac{7}{11}$ and  Theorem \ref{TW_GLOBAL} may be applied  for this value of parameter $\kd$.
\end{rem}

\no As a consequence of theorem~\ref{TW_GLOBAL} we have
\begin{coro}
Assume that $\kd > \frac{1}{2}$, $v_{0}\in \Vkd^{2}$,  $\om_{0}$, $b_{0}\in \V^{2}$   and the conditions (\ref{f}), (\ref{g}) hold. We denote
\[
a_{0}=\sup_{t\geq 0 } 2 \comkd \jkt^{\frac{1}{\kd}-1} \left(A(t)+ B(t)Y_{2}^{\frac{1}{4}}(t) + C(t)Y_{2}^{\frac{3}{4}}(t)+ D(t)Y_{2}^{\frac{5}{4}}(t)  \right),
\]
where $\comkd $ is the constant given in theorem~\ref{TW_GLOBAL} and $Y_{2}, A(t),\dots, D(t)$ were defined in (\ref{Y_0t})-(\ref{def_D}). Then $a_{0}$ is finite. If in addition,
\eqnsl{
\frac{\bmi}{\oma} > 2 \comkd \n{ \| b_{0}\|_{1} + \jd \ndk{ v_{0}}\n{1 + I_\infty \n{\kd, \frac{\omi}{\oma}, \frac{\bmi}{(\oma)^2}}} } \\
\text{ for } \kd \geq 1
}{Z1}
and
\eqnsl{
\frac{\bmi}{\oma} > 2 \comkd \n{ \| b_{0}\|_{1} + \jd \ndk{ v_{0}}\n{1 + I_\infty \n{\kd, \frac{\omi}{\oma}, \frac{\bmi}{(\oma)^2}}} } \n{\frac{\oma}{\omi}}^\frac{1}{\kd} \\
\text{ for } \kd \in \n{\jd , 1}
}{Z1.5}
and
\eqq{\frac{\bmi}{\oma} > a_{0} \left( \ndk{\lap v_{0}}+ \ndk{\lap \om_{0}}+ \ndk{\lap b_{0}}   \right)^{\frac{1}{4}}}{Z2}
hold, then the system (\ref{a})-(\ref{e}) has a unique global solution in $\mathcal{X}(\infty)$.
\label{coro_glob}
\end{coro}

\begin{rem}
The conditions (\ref{Z1})-(\ref{Z2}) involve only the initial data: $v_{0}$, $\om_{0}$, $b_{0}$,  the parameters of the system: $\nu_{0}$, $\kappa_{1}$, \dots , $\kappa_{4}$ and $\Om$.
\end{rem}

\begin{rem}
We shall show  that the conditions (\ref{Z1})-(\ref{Z2}) are satisfied on some non-empty set of initial data. We focus only on the case  $\kd\in (\jd, 1)$, because the other is simpler. It may be done in the following way: we shall determine positive $\dlj$,$\dld$,$\dlt$ such that if initial data  satisfy the bounds
\eqq{
\| b_{0}\|_{1}\leq \dlj, \hd \| v_{0}\|_{2}\leq \dld, \hd \lapy \leq \dlt,
}{delty}
then (\ref{Z1.5}) and (\ref{Z2}) will be fulfilled. We proceed the following steps

\begin{itemize}
\item
set $\omi$ and $\oma$ such that $0<\omi<\oma$ and
\[
2 \comkd |\Om| (\oma)^{1+\frac{1}{\kd}}  <(\omi)^{\frac{1}{\kd}} ,
\]
i.e.
\[
\frac{1}{\oma}> 2 \comkd |\Om| \n{\frac{\oma}{\omi}}^{\frac{1}{\kd}},
\]
\item
fix  $\bmi>0$ so, we have
\[
\frac{\bmi}{\oma} > 2 \comkd  \bmi |\Om|\n{\frac{\oma}{\omi}}^{\frac{1}{\kd}},
\]
\item
choose $\dlj>\bmi |\Om|$ such that
\[
\frac{\bmi}{\oma} > 2 \comkd  \dlj \n{\frac{\oma}{\omi}}^{\frac{1}{\kd}},
\]
\item
find $\dld>0$ such that
\[
\frac{\bmi}{\oma} > 2 \comkd \n{ \dlj + \jd \dld \n{1 + I_\infty \n{\kd, \frac{\omi}{\oma}, \frac{\bmi}{(\oma)^2}}} } \n{\frac{\oma}{\omi}}^\frac{1}{\kd} ,
\]
\item
if we define $a_{0}(\dlj, \dld, \dlt)$ similarly as  in Corollary~\ref{coro_glob}, where we replace $\| b_{0}\|_{1}$ by $\dlj$, $\| v_{0}\|_{2}$ by $\dld$ and  $\lapy$  by $\dlt$, then from (\ref{bMinMaxt}), (\ref{Y_0t}) and (\ref{def_A})-(\ref{def_C}) we deduce that $a_{0}(\dlj, \dld, \dlt)$ is increasing with respect to each $\dl_{i}$. Therefore, we can find $\dlt>0$ such that
\[
\frac{\bmi}{\oma}> a_{0}(\dlj, \dld, \dlt) \dlt^{\frac{1}{4}},
\]
\item
finally, for these positive numbers $\dlj, \dld, \dlt$  and any $b_{0}$, $\om_{0}$ and $v_{0}$ such that $\bmi\leq b_{0}$, $\omi\leq \om_{0} \leq \oma$ and (\ref{delty}) hold, the conditions  (\ref{Z1.5}) and (\ref{Z2}) are satisfied.
\end{itemize}

\end{rem}


\section{Proof of Theorem~\ref{TW_GLOBAL}}

We need the following auxiliary results (see also theorem 4.1 \cite{MiNa}).

\begin{prop}
Assume that  $\om_{0}$, $b_{0}\in \V^{2}$,  $v_{0}\in \Vkd^{2}$ and (\ref{f}), (\ref{g}) hold. If $T>0$ and $(v, \om, b)\in \mathcal{X}(T) $ satisfies (\ref{a})-(\ref{e}), then the following estimates
\eqq{\frac{\omi}{1+\kd \omi t} \leq \om(x,t) \leq \frac{\oma}{1+\kd \oma t},}{newcog}
\eqq{\frac{\bmi}{(1+\kd \oma t)^{\frac{1}{\kd}}}  \leq b(x,t) }{newdog}
hold for $(x,t)\in \Om^{T}$.
\label{oszac_og}
\end{prop}

\begin{proof}
By assumption we have $\om, b \in L^{2}_{loc}([0,T);H^{3}(\Om))$, $\om_{,t}, b_{,t} \in L^{2}_{loc}([0,T);H^{1}(\Om))$ thus, Sobolev embedding theorem implies that $\om, b\in C(\overline{\Om}\times [0,T))$. Then, by (\ref{f}) and (\ref{g})  there exists $t_{1}\in (0,T)$ such that
\eqq{\jd \bmi \leq b(x,t), \hd \jd \omi \leq \om(x,t) \leq 2 \oma \hd \m{ for } \hd (x,t)\in \Om^{t_{1}}. }{dpa}
We denote by $f_{+}$ and $f_{-}$ the non-negative  and non-positive parts of function $f$, i.e. $f= f_{+}+f_{-}$, where $f_{+}=\max\{f,0 \}$. For $t\in (0,t_{1})$ we test the equality (\ref{ab}) by $z=\ommm$ and we obtain
\[
(\om_{,t}, \ommm)+ \n{\bom \nab \om , \nab \ommm} = - \kd (\om^{2}, \ommm),
\]
where we used the condition $\divv v =0$. Using the equality $(\omt)_{,t}= - \kd (\omt)^{2}$ we may write
\[
\jd \ddt \ndk{ \ommm} - \kd \n{ (\omt)^{2}, \ommm  }+\n{ \bom \nab \ommm, \nab \ommm  }
\]
\[
= - \kd ( \om^{2}, \ommm)
\]
for $t\in (0, t_{1})$. After applying (\ref{dpa}) we get
\[
\jd \ddt \ndk{ \ommm} \leq - \kd \n{ (\om - \omt)(\om+\omt), \ommm }
\]
\[
= - \kd \n{ \om+\omt, \bk{ \ommm} }.
\]
By Gr\"onwall inequality and (\ref{g}) we deduce that $\ommm\equiv 0$ on for $t\in (0,t_{1})$ hence
\eqq{\frac{\omi}{1+\kd \omi t} \leq \om(x,t)}{omnatj}
for $(x,t)\in \overline{\Om}\times [0,t_{1})$. Next, if we test the equation (\ref{ab}) by $z=\ommp$, then proceeding similarly we deduce that
\eqq{\om(x,t)\leq \frac{\oma}{1+\kd \oma t} }{ommnatj}
for $(x,t)\in \overline{\Om}\times [0,t_{1})$.
Now, for $t\in (0,t_{1})$ we test the equation (\ref{ac}) by $q=\bbtmm$ and we obtain
\[
(b_{,t}, \bbtmm) +\n{ \bom \nab \bbtmm, \nab \bbtmm }
\]
\[
=- (b\om , \bbtmm )+ \n{ \bom \bk{D(v)}, \bbtmm},
\]
where we used the condition $\divv v =0$. By applying (\ref{dpa}) we get
\[
(b_{,t}, \bbtmm) \leq - (b\om , \bbtmm ),
\]
i.e.
\[
\jd \ddt \ndk{ \bbtmm} -  \frac{\oma  }{\jkt }
\n{ \bmt, \bbtmm   } \leq -(b\om, \bbtmm).
\]
From (\ref{dpa}) and (\ref{ommnatj}) we get
\[
-(b\om, \bbtmm) \leq - \frac{\oma}{\jkt}(b, \bbtmm )
\]
for $t\in (0,t_{1})$ hence, we obtain
\[
\jd \ddt \ndk{ \bbtmm}\leq -\frac{\oma}{\jkt}(b-\bmt, \bbtmm ) .
\]
The right-hand side in non-positive thus, we from (\ref{f}) have
\eqq{\frac{\bmi}{\jkt^{\frac{1}{\kd}}} \leq b(x,t) }{bmtj}
for $(x,t)\in \overline{\Om}\times [0,t_{1})$.
Now, we define
\[
\tjs = \sup\{\widetilde{t}\in (0,T): (\ref{newcog}), (\ref{newdog}) \m{ \hd hold  for  } (x,t)\in \Om^{\widetilde{t}} \}.
\]
By the previous step we have $\tjs\geq t_{1}>0$. If $\tjs<T$, then by continuity of $\om, b$ and (\ref{omnatj})-(\ref{bmtj}) there exists $t_{2}\in (\tjs, T)$ such that
\[
\jd \bmt \leq b(x,t), \hd \hd \jd \omt\leq \om (x,t) \leq 2\ommt \hd \m{ for } (x,t)\in \Om^{t_{2}}.
\]
Then, we have $\frac{b(x,t)}{\om(x,t)}\geq \frac{1}{4}\frac{\bmt}{\omt}>0$ for $(x,t)\in \Om\times [0,t_{2})$ and we may repeat the argument from the first part of the proof and as a consequence we get $t_{2} \le \tjs$. This contradiction means that $\tjs=T$ and the proof is finished.
\end{proof}


\begin{prop}
For any $T>0$, the problem (\ref{a})-(\ref{e}) has at most one solution in $\mathcal{X}(T)$.
\label{jed}
\end{prop}

\begin{proof}
Suppose that $(\vj,\oj,\bj )$, $(\vd, \od , \bd)\in \mathcal{X}(T)$ satisfy (\ref{a})-(\ref{e}) in $\OmT$. We denote $v=\vj - \vd $, $\om = \oj - \od$, $b= \bj - \bd$ and we test the equations for $\vj$ and $\vd$ by $v$. After subtracting the equations for $v^{i}$ we get
\[
(v_{,t}, v )- \n{ \vj \otimes \vj  - \vd \otimes \vd , \nab v  }+ \n{ \bjoj D(\vj) - \bdod D(\vd), D(v)    }=0.
\]
We note that
\[
\n{ \bjoj D(\vj) - \bdod D(\vd), D(v)    }
\]
\[
= \n{\bjoj D(v), D(v)  }+ \n{\boj, D(\vd), D(v)  }-\n{ \frac{b^{2} \om }{\oj \od }  D(v^{2}), D(v)  },
\]
\[
\n{ \vj \otimes \vj  - \vd \otimes \vd , \nab v  } = \n{\vj \otimes v, \nab v   } + \n{ v\otimes \vd , \nab v  }.
\]
By proposition~\ref{oszac_og} we have $\bjoj\geq \mnit$ thus, by H\"older inequality we get
\[
\jd \ddt \ndk{v}+ \mnit \ndk{ D(v)} \leq \nif{\joj}\nd{ b} \nif{ D(\vd)} \nd{ D(v)}
\]
\[
+ \nif{ \jojod  } \nif{ \bd} \nd{\om} \nif{D(\vd)} \nd{D(v)} + \nif{ \vj} \nd{v} \nd{ \nab v } + \nd{v} \nif{ \vd} \nd{ \nab v }.
\]
By proposition~\ref{oszac_og} functions $\oj$ and $\od$ are estimated from below by $\omt$ hence, if we apply Young inequality, Sobolev embedding theorem and $\nd{D(v)} = \frac{\sqrt{2}}{2} \nd{\nabla v}$, then we obtain
\[
\ddt \ndk{v}+ \mnit \ndk{ D(v)}
\]
\eqq{ \leq \frac{C}{\mnit } \n{ (\omt)^{-2} \nsotk{\vd} \ndk{ b}+ (\omt)^{-4} \nsodk{\bd}\nsotk{\vd} \ndk{\om} + \n{\nsodk{\vj}+\nsodk{ \vd} }\ndk{v}  },    }{jed_a}
where $C$ depends only on $\Om$. Now, we test the equations for $\oj$ and $\od$ by $\om = \oj- \od$ and as a result we obtain
\[
\jd \ddt \ndk{ \om }+  \n{  \bjoj \nab \om, \nab \om  } = \n{ \oj v, \nab \om    }+ \n{ \om \vd , \nab \om   }
\]
\[
-\n{ \boj \nab \od, \nab \om } + \n{  \frac{\bd \om}{\oj \od} \nab \od, \nab \om  } - \kd \n{  \om (\oj+ \od), \om   }.
\]
From H\"older inequality and (\ref{newcog}) we get
\[
\jd \ddt \ndk{ \om} + \mnit \ndk{ \nab \om }
\]
\[
 \leq  \nif{ \oj } \nd{ v} \nd{ \nab \om} + \nd{  \om } \nif{\vd} \nd{ \nab \om} + \nif{\joj} \nd{ b} \nif{\nab \od} \nd{ \nab \om}
\]
\[
 + \nif{ \jojod } \nif{ \bd} \nd{\om} \nif{ \nab \od} \nd{ \nab \om} + \kd \nif{ \oj + \od } \ndk{\om}.
\]
By Young inequality and Sobolev embedding theorem we obtain
\[
\ddt \ndk{ \om} + \mnit \ndk{ \nab \om } \leq \frac{C}{\mnit} \Big( \nsodk{ \oj} \ndk{ v}
\]
\eqnsl{ + \n{ \nsodk{ \vd } + (\omt)^{-4} \nsodk{\bd} \nsotk{ \od} + \mnit\nsodk{\oj} + \mnit\nsod{\od}    }\ndk{ \om} \\ + (\omt)^{-2} \nsotk{\od} \ndk{b}   \Big),
}{jed_b}
where $C$ depends only on $\Om$ and $\kd$. Finally, we test the equations for $\bj$ and $\bd$ by $b= \bj- \bd$ and we get
\[
\jd \ddt \ndk{ b} + \n{ \bjoj \nab b, \nab b }= \n{\bj v, \nab b  }+ \n{ b \vd , \nab b} -\n{ \boj \nab \bd, \nab b } + \n{  \frac{\bd \om}{\oj \od} \nab \bd, \nab b  }
\]
\[
-\n{\bj \om , \nab b   } - \n{ b \od, \nab b  } + \n{ \bjoj \bk{D(\vj) } -\bdod \bk{ D(\vd)} , b   }.
\]
We note that the last term on the right-hand side is equal to
\[
\n{ \bjoj D(v)  D(\vj + \vd), b  } + \n{ \boj \bk{ D(\vd)} , b  } - \n{ \frac{ \bd \om }{\oj \od} \bk{ D(\vd)}, b }.
\]
From H\"older inequality and (\ref{newcog}), (\ref{newdog}) we get
\eqns{\jd \ddt \ndk{ b} + & \mnit \ndk{ \nab b } \leq \nif{ \bj} \nd{ v} \nd{ \nab b } + \nd{ b} \nif{ \vd} \nd{\nab b } \\
& + \nif{ \joj} \nd{ b} \nif{\nab \bd} \nd{ \nab b} + \nif{\jojod} \nif{ \bd } \nd{ \om} \nif{\nab \bd} \nd{\nab b } \\
& + \nif{ \bj} \nd{ \om }\nd{ \nab b } + \nd{b} \nif{ \od} \nd{\nab b } \\
& + \nif{ \joj } \nif{ \bj } \nd{ D(v) } \nif{ D(\vj + \vd)} \nd{b} \\
& + \nif{ \joj} \ndk{ b} \nifk{D(\vd)} +\nif{ \jojod} \nif{\bd} \nd{\om} \nifk{D(\vd)} \nd{b}.
 }
Applying Young inequality and Sobolev embedding theorem we obtain
\eqnsl{\ddt \ndk{ b} +   \mnit \ndk{\nab b } & \leq  \frac{C}{\mnit} \Big\{ \nsodk{\bj }\ndk{v} \\
& + \Big[ \nsodk{\vd} + (\omt)^{-2} \nsotk{\bd}+ \nsodk{ \od} \\
& +(\omt)^{-2}  (\mnit + \nsodk{ \bj})( \nsotk{ \vj }+ \nsotk{ \vd}) \\
& + \mnit(\omt)^{-1} \nsotk{ \vd}  \Big] \ndk{ b} \\
& +\Big[ (\omt)^{-4}\nsodk{ \bd}   \nsotk{ \bd} +\nsodk{b^{1}} \\
& + \mnit(\omt)^{-2} \nsodk{b^{2}} \nsotk{ \vd} \Big] \ndk{ \om}   \Big\} + \mnit \ndk{ D(v)}.
}{jed_c}
If we sum the inequalities (\ref{jed_a})-(\ref{jed_c}), then we obtain
\[
\ddt \n{ \ndk{v}+ \ndk{\om} + \ndk{ b}   } \leq h(t) \n{ \ndk{v}+ \ndk{\om} + \ndk{ b}   },
\]
with $h\in L^{1}(0,T)$, because after applying the embedding $\mathcal{X}(T)  \hookrightarrow L^{\infty}(0,T;H^{2}(\Om))  $ we deduce that
$(v^{i}, \om^{i}, b^{i})$ belong to $L^{\infty}(0,T;H^{2}(\Om))\cap L^{2}(0,T;H^{3}(\Om))$ (the embedding is just a consequence of one integration by parts in the term $\ddt \| \Delta u \|_{2}^{2}$).  By the assumption, $v(0)=0$, $\om(0)=0$, $b(0)=0$ thus, by Gr\"onwall inequality we get $v\equiv 0$, $\om \equiv 0$ and $b\equiv 0$ on $\Om^{T}$ and the proof is finished.
\end{proof}


Suppose that the assumptions of theorem~\ref{TW_GLOBAL} hold. Then, by theorem \ref{LOCALNE_TH} there exists regular, local in time solution to the system (\ref{a})-(\ref{e}), which belongs to $\mathcal{X}(T_{0})$ for some positive $T_{0}$. From Proposition~\ref{jed} it is unique solution in  $\mathcal{X}(T_{0})$. We will show that provided the  smallness condition imposed on initial data (formulated in  (\ref{GLOB_ADD})), the solution exists on $[0,T)$. In particular, if (\ref{GLOB_ADD}) holds with  $T=\infty$, then the solution is global, i.e. it belongs to $\mathcal{X}(\infty)$. Firstly, we denote
\eqq{
T^{*}=\sup\{t^{*}>0: \hd \m{ system (\ref{a})-(\ref{e}) has a solution  } (v,\om , b) \m{ \hd in } \mathcal{X}(t^{*}) \}.
}{defTgw}
We note that  $T^{*}\geq T_{0}>0$. By Proposition~\ref{jed} there exists  $(v,\om, b)$  the  unique solution of (\ref{a})-(\ref{e}) in $\mathcal{X}(T^{*})$, i.e. the following identities
\eqq{(v_{,t}, w) -  (v\otimes v,\nabla w)  +     \n{ \frac{b}{\om}  D(v), D(w)} = 0 \hd \m{ for } \hd w\in \Vkdj}{gbh}
\eqq{ (\om_{,t}, z) - (\om v, \nabla z  ) + \n{ \frac{b}{\om} \nabla \om, \nabla z  } = - \kd (\om^{2}, z ) \hd \m{ for } \hd z\in \Vj,  }{cbgh}
\eqq{(b_{,t}, q) - (b v, \nab q   )   +\n{  \frac{b}{\om}  \nabla b, \nabla q  }  = - (b \om , q) + \left( \frac{b}{\om} \bk{D(v)}, q\right) \hd \m{ for } \hd q\in \Vj ,  }{ccgno}
hold for a.a. $t\in (0,\Ts)$, where $(\cdot,\cdot)$ denotes inner product in $L^{2}(\Om)$. By Proposition~~\ref{oszac_og} functions $\om$ and $b$ satisfy
\eqnsl{
b(t,x)\geq \bmt,  \hd \hd
\om(t,x) \geq \omt, \hd \hd
\om(t,x) \le \ommt \hd \hd  \m{for } \hd (x,t)\in \Om^{\Ts}.
}{inf-EST-b-omega}

We shall show that if the condition (\ref{GLOB_ADD}) holds for some $T$, then $T^{*}\geq T$. As it will be explained in the proof of Corollary~\ref{coro_glob}, the condition  (\ref{GLOB_ADD}) holds, provided the initial data are sufficiently small.

To prove the result we suppose that $T^{*}<T$ and we shall show that it leads to a contradiction. The idea of the proof is as follows: we shall show that under smallness assumption $(\ref{GLOB_ADD})$ we are able to obtain an estimate for solution in $H^{2}(\Om)$ norm, which is uniform with respect to $t\in [0,\Ts)$. Next, by applying Theorem~\ref{LOCALNE_TH} and Proposition~\ref{oszac_og} we will be able to extend the solution beyond $\Ts$ and this is a contradiction with the definition of $\Ts$. Therefore, the key step in the proof is to get the estimates in $H^{2}$  norm  for solution $(v,\om,b)$. First we deal with the lower order terms.

\subsection{The lower order estimates}

In this subsection we estimate  the $L^2$-norm of $v$ and next, the $L^{1}$-norm of $b$. The proof of the main theorem depends heavily on the decay  estimates  of these quantities. In the proposition below we consider all values of $\kd\in (0,\infty)$ to illustrate the influence of $\kd$ for the available decay estimates. From this we will see that $\kd=\jd$ seems to be critical value.

\begin{prop} \label{propEstimates}
For each $t\in [0,\Ts)$ the following estimates holds
\begin{enumerate}[a)]
\item
\eqnsl{
\nd{ v(t) } \le \nd{ v_0 } \exp \n{ - \frac{1}{C^{2}_{p}}\frac{\bmi }{ \oma^{2} \n{2 \kd - 1 }}  \n{\n{1 + \kd \oma t }^{2 - \frac{1}{\kd}}-1} } \\
\text{for } \kd \in \n{0,\jd} \cup \n{\jd, \infty} ,
}{noa}
and
\eqnsl{\nd{ v(t)} \leq  \nd{ v_0 }(1+\kd \oma t )^{-\frac{\bmi}{C_{p}^2\oma^2 \kd} } \hd \m{ for } \hd \kd =\frac{1}{2},}{estivk2jd}
\item
\eqq{
\nd{\om(t)} \leq \nd{\om_{0}} \hd
\text{for } \kd \in \n{0,\infty},
}{ldomega}
\item
\eqnsl{
\nj{ b(t)}&  +   \jd \ndk{v(t)} \\
& \le \frac{\nj{ b_0 } + \jd \ndk{ v_0 }\n{1+I_\infty\n{\kd, \frac{\omi}{\oma}, \frac{\bmi}{\n{\oma}^2}}}}{ \n{1 + \kd \omi t}^{\frac{1}{\kd}}}
\hd
\text{for } \kd \in \n{\jd, \infty},
}{b-l1-est}
\item
\eqnsl{
   \frac{\nj{ b_0 } + \jd \ndk{ v_0 }}{ \n{1 + \kd \oma t}^{\frac{1}{\kd}}} \leq \nj{ b(t) } + \jd \ndk{v(t)}
\hd
\text{for } \kd \in \n{0, \infty},
}{kinEnergyEstFromBelow}
\item
\eqnsl{
\nj{b} + \ndk{ v(t)}
\le \frac{\nj{b_0} + \ndk{ v_0}}{\n{1 + \kd \omi t}^{\frac{1}{\kd} \min \left\{1,\frac{C_p^2 \oma^2}{\bmi} \right\}}}
\text{\hd for } \kd \in \n{\jd, \infty},
}{b-l1-est2}
\item
\eqnsl{
\nj{b} + \jd \ndk{ v(t)}
\le \nj{b_0} + \jd \ndk{ v_0}
\text{ \hd for } \kd \in \n{0, \infty},
}{b-l1-est3}
\end{enumerate}
where $I_\infty $ was defined in (\ref{IinfDef}), hold.
\end{prop}


\begin{proof}[Proof of Proposition \ref{propEstimates}]
{\bf a)}
We test the equation (\ref{gbh}) by $v$ and  we get
\eqnsl{
\jd  \ddt \ndk{ v } + \n{\frac{b}{\omega} D(v), D(v)}=0 \hd \m{ for }\hd t\in (0,\Ts),
}{temp-v-vmean-lower-eq}
where we applied the condition $\divv{v}=0$. Using the notation (\ref{mnitDef}) and  the estimate (\ref{inf-EST-b-omega}) we obtain
\eqns{
\jd  \ddt \ndk{ v } + \mnit \ndk{D(v)} \le 0  \hd \m{ for }\hd t\in (0,\Ts).
}
The mean value of components of $v$ are zero thus, from the Poincar\'e inequality and the fact that $\nd{D(v)} = \frac{\sqrt{2}}{2} \nd{\nabla v}$ we get
\[
\jd  \ddt \ndk{ v } + \mnit \frac{1}{C^{2}_{p}} \ndk{v } \le 0  \hd \m{ for }\hd t\in (0,\Ts)      .
\]
By applying (\ref{mnitDef}) we may write explicitly
\eqnsl{
\ddt \ndk{ v(t) } +  \frac{2}{C^{2}_{p}}  \frac{\bmi}{ \oma} \n{1 + \kd \oma t }^{1 - \frac{1}{\kd}}\ndk{ v(t)}\leq 0  \hd \m{ for }\hd t\in (0,\Ts) .
}{v-l2-norm-ineq}
Multiplying by appropriate exponential function, after integration we obtain (\ref{noa}) and (\ref{estivk2jd}).

{\bf b)}  If we test the equation (\ref{cbgh}) by $z=\om$, then after integration by parts and using (\ref{inf-EST-b-omega}) we get
\[
\jd \ddt \ndk{ \om(t)} \leq 0 \hd \m{ for } t\in (0,\Ts)
\]
thus, we have (\ref{ldomega}).

{\bf c)}
We now proceed to estimate for $b$. We can not obtain any pointwise estimate from above for $b$. However, we are able to estimate the $L^1$-norm of $b$.  Indeed, we test  the  equation (\ref{ccgno}) by $q\equiv 1$ and we get
\eqns{
\n{ \bt,  1 }  = - \n{ b \om, 1 }
+  \n{\frac{b}{\omega} \dvk, 1 }
}
The positivity of $b$ follows from (\ref{f}), (\ref{bMinMaxt}) and (\ref{inf-EST-b-omega}) so,  we get
\[
\ddt \nj{ b }  = - \n{ b \om, 1 }
+  \n{\frac{b}{\omega} \dvk, 1 }.
\]
We note that the term $\n{\frac{b}{\omega} \dvk, 1 } $ is equal to $\n{\frac{b}{\omega} D(v), D(v)}$ thus, we can use the equation (\ref{temp-v-vmean-lower-eq}) and we obtain
\eqnsl{
\ddt \nj{ b }  = - \n{ b \om, 1 }
- \jd \ddt  \ndk{v}.
}{b-l1-norm-eq}
From (\ref{omMinMaxt}) and (\ref{inf-EST-b-omega}) we may estimate  $\omega$ from below and  we obtain
\eqnsl{
\ddt \nj{ b }  \le - \omitt \nj{ b} - \jd \ddt \ndk{ v }.
}{b-l1-norm-ineq}
By multiplying both sides by $e^{\int_0^t \omittau d \tau}$ we get
\eqns{
\ddt \n{ \nj{ b } e^{\int_0^t \omittau d \tau} }
\le
- \jd \ddt \ndk{ v } e^{\int_0^t \omittau d \tau}.
}
After integrating from $0$ to $t$ we get
\eqns{
 \nj{ b } e^{\int_0^t \omittau d \tau}
\le
\nj{ b_0 } - \jd \int_0^t \frac{d}{d \tau} \ndk{ v(\tau) } e^{\int_0^\tau \omits ds} d\tau.
}
After integrating by parts we get
\eqns{
 \nj{ b } e^{\int_0^t \omittau d \tau}
& \le
\nj{ b_0 } -  \left[\jd \ndk{ v(\tau) } e^{\int_0^\tau \omits ds} \right]_{\tau = 0}^{\tau = t}   \\
&  + \jd \int_0^t  \ndk{ v(\tau) } e^{\int_0^\tau \omits ds} \omittau d\tau.
}
Thus we get
\eqns{
 \nj{ b }  + \jd \ndk{v}
& \le
\n{ \nj{ b_0 } + \jd \ndk{ v_0 } } e^{-\int_0^t \omittau d \tau}  \\
&  + \jd e^{-\int_0^t \omittau d \tau} \int_0^t  \ndk{ v(\tau) } e^{\int_0^\tau \omits ds} \omittau d\tau.
}
We note that
\eqns{
\int_0^t \omittau d \tau = \ln \n{1 + \kd \omi t}^{\frac{1}{\kd}}
}
thus, we obtain
\eqns{
 \nj{ b }   + \jd \ndk{v}
& \le
\frac{ \nj{ b_0 } + \jd \ndk{ v_0 } }{\n{1 + \kd \omi t}^{\frac{1}{\kd}}} \\
&  + \jd \frac{1}{\n{1 + \kd \omi t}^{\frac{1}{\kd}}} \int_0^t  \ndk{ v(\tau) } \frac{\omi}{\n{1 + \kd \omi \tau}^{1 - \frac{1}{\kd}}} d\tau.
}
After using (\ref{noa}) we get
\eqns{
 \nj{ b }   + \jd \ndk{v}
& \le
\frac{ \nj{ b_0 } + \jd \ndk{ v_0 } }{\n{1 + \kd \omi t}^{\frac{1}{\kd}}}
 +  \frac{\jd \ndk{v_0} }{\n{1 + \kd \omi t}^{\frac{1}{\kd}}} I_t(\kd, \omi, \oma, \bmi),
}
where
\eqnsl{
I_t(\kd , &\omi, \oma, \bmi) \\
& = \int_0^t
\exp \n{- \frac{2\bmi \n{\n{1 + \kd \oma \tau }^{2 - \frac{1}{\kd}}-1} }{C^{2}_{p} \oma^{2} \n{2 \kd - 1 }}   }
\frac{\omi}{\n{1 + \kd \omi \tau}^{1 - \frac{1}{\kd}}}
d \tau.
}{ItEst:1}
Now, we  shall obtain  an estimate of $I_t$. Depending on the value of $\kd$, we obtain different types  of the estimates. Firstly, we focus  on the  case $\kd \geq 1$. From  (\ref{newcog}) we have
\eqns{
\frac{\omi}{\n{1 + \kd \omi \tau}^{1 - \frac{1}{\kd}}}
& =
\frac{\n{\omi}^{\frac{1}{\kd}} \n{\omi}^{1 - \frac{1}{\kd}}}{\n{1 + \kd \omi \tau}^{1 - \frac{1}{\kd}}}
 \le
 \frac{\n{\omi}^{\frac{1}{\kd}} \n{\oma}^{1 - \frac{1}{\kd}}}{\n{1 + \kd \oma \tau}^{1 - \frac{1}{\kd}}}
}
and thus
\eqnsl{
& I_t(\kd , \omi, \oma, \bmi) \\
& \le \oma \n{\frac{\omi}{\oma}}^{\frac{1}{\kd}} \int_0^t
\exp \n{- \frac{2\bmi \n{\n{1 + \kd \oma \tau }^{2 - \frac{1}{\kd}}-1} }{C^{2}_{p} \oma^{2} \n{2 \kd - 1 }}   }
\frac{d \tau}{\n{1 + \kd \oma \tau}^{1 - \frac{1}{\kd}}}
.
}{ItEst1}
Now, we can change variables $s = 1 + \kd \oma t$ and we have
\eqns{
& I_t(\kd , \omi, \oma, \bmi) \\
& \le
\frac{1}{\kd}\n{\frac{\omi}{\oma}}^{\frac{1}{\kd}}
\exp \n{\frac{2\bmi}{C^{2}_{p} \oma^{2} \n{2 \kd - 1 }}   }
\int_1^\infty
\exp \n{- \frac{2\bmi s^{2 - \frac{1}{\kd}} }{C^{2}_{p} \oma^{2} \n{2 \kd - 1 }}   }
s^{\frac{1}{\kd}-1}
ds.
}
Next, the change of variables $y=\frac{2\bmi s^{2 - \frac{1}{\kd}} }{C^{2}_{p} \oma^{2} \n{2 \kd - 1 }}  $ leads to the estimate
\eqns{
& I_t(\kd , \omi, \oma, \bmi) \\
& \le
\n{\frac{\omi}{\oma}}^{\frac{1}{\kd}}
\exp \n{\frac{2\bmi}{C^{2}_{p} \oma^{2} \n{2 \kd - 1 }}   }
\n{\frac{C_p^2 (\oma)^2 (2\kd-1)}{2\bmi}}^{\frac{2\kd}{2\kd -1}}
\Gamma\left(\frac{1}{2\kd-1}\right).
}
Thus, in the case of $\kd \geq 1$ we  obtain
\eqns{
& \nj{ b }     + \jd \ndk{v}
 \\
& \le
\frac{ \nj{ b_0 } + \jd \ndk{ v_0 }
\n{
1 +
\Gamma\n{\frac{2\kd}{2\kd - 1}}
\n{\frac{\omi}{\oma}}^{\frac{1}{\kd}}
\n{\frac{C_p^2 (\oma)^2(2\kd -1)}{2\bmi}
\exp \n{\frac{2\bmi}{C^{2}_{p} \oma^{2}}   }
}^{\frac{1}{2\kd -1}}
}
}{\n{1 + \kd \omi t}^{\frac{1}{\kd}}}
}
hence, (\ref{b-l1-est}) holds for $\kd \geq 1$. Now, if we assume that
 $\kd \in \n{\jd,1}$, then  we have
\eqns{
\frac{1}{\n{1 + \kd \omi \tau}^{1 - \frac{1}{\kd}}}
\le
\frac{1}{\n{1 + \kd \oma \tau}^{1 - \frac{1}{\kd}}}
}
and  from (\ref{ItEst:1}) we obtain
\eqnsl{
I_t(\kd , &\omi, \oma, \bmi)
\\ & \le
\omi \int_0^t
\exp \n{- \frac{2\bmi \n{\n{1 + \kd \oma \tau }^{2 - \frac{1}{\kd}}-1} }{C^{2}_{p} \oma^{2} \n{2 \kd - 1 }}   }
\frac{d \tau}{\n{1 + \kd \oma \tau}^{1 - \frac{1}{\kd}}}.
 }{ItEst2}
Proceeding as earlier   we obtain
\eqns{
& \nj{ b }    + \jd \ndk{v} \\
& \le
\frac{ \nj{ b_0 } + \jd \ndk{ v_0 }
\n{
1 +
\Gamma\n{\frac{2\kd}{2\kd - 1}}
\frac{\omi}{\oma}
\n{\frac{C_p^2 (\oma)^2(2\kd-1)}{2\bmi}
\exp \n{\frac{2\bmi}{C^{2}_{p} \oma^{2}}   }
}^{\frac{1}{2\kd -1}}
}
}{\n{1 + \kd \omi t}^{\frac{1}{\kd}}},
}
hence, (\ref{b-l1-est}) also holds for $\kd \in \n{\jd , 1}$.

{\bf d) } Now, we shall obtain (\ref{kinEnergyEstFromBelow}) - the estimate from below. Firstly, we note that from (\ref{newcog})  and (\ref{b-l1-norm-eq}) we have
\eqns{
\ddt \n{\nj{ b} + \jd \ndk{ v }}  \geq - \omat \nj{ b}
}
hence, we get
\eqns{
\ddt \ln \n{\nj{ b} + \jd \ndk{ v }}  \geq - \omat.
}
After integration both sides from $0$ to $t$ we obtain
\eqns{
\ln \n{ \frac{\nj{ b} + \jd \ndk{ v }}{\nj{ b_0 } + \jd \ndk{ v_0 }}}
\geq
-\frac{1}{\kd} \ln \n{1 + \kd \oma t}
}
so, the inequality  (\ref{kinEnergyEstFromBelow}) is proved.

{\bf e) } Now, we shall prove (\ref{b-l1-est2}). From   (\ref{v-l2-norm-ineq}) and (\ref{b-l1-norm-ineq}) we have
\eqnsl{
\ddt \n{ \nj{b} + \ndk{ v } }
+ \omitt \nj{b}
+  \frac{1}{C^{2}_{p}}  \frac{\bmi}{ \oma} \n{1 + \kd \oma t }^{1 - \frac{1}{\kd}}\ndk{ v}
\leq 0.}{tmp::1}
We shall show   that  for $C_0 = \frac{C_p^2 \oma^2}{\bmi}$ and $t\geq 0$ there holds
\eqnsl{
\omitt \le
\frac{C_0}{C^{2}_{p}}  \frac{\bmi}{ \oma} \n{1 + \kd \oma t }^{1 - \frac{1}{\kd}}.
}{C0InEq0}
Indeed, it  is equivalent to
\eqnsl{
1 \le
\frac{C_0}{C^{2}_{p}}  \frac{\bmi}{\omi \oma} \n{1 + \kd \omi t } \n{1 + \kd \oma t }^{1 - \frac{1}{\kd}}
}{C0InEq}
and  for $\kd \geq 1$ the  right-hand side is increasing function of $t$, so it is enough to check (\ref{C0InEq}) for $t=0$, which is obviously true. If $\kd \in \n{\jd,1}$, then $\frac{1}{\kd}-1$ and $2-\frac{1}{\kd}$ are positive and the function    $\frac{1 + \kd \omi t}{1 + \kd \oma t }$ is monotonically decreasing, where  $\lim\limits_{t \rightarrow \infty} \frac{1 + \kd \omi t}{1 + \kd \oma t } = \frac{\omi}{\oma}$ thus, we have
\eqns{
\frac{C_0}{C^{2}_{p}}
\frac{\bmi}{\omi \oma} &
\n{1 + \kd \omi t }^{2 - \frac{1}{\kd}} \n{\frac{1 + \kd \omi t}{1 + \kd \oma t }}^{\frac{1}{\kd} - 1} \\
& \geq \frac{C_0}{C^{2}_{p}}
\frac{\bmi}{\omi \oma} \n{ \frac{\omi}{\oma}}^{\frac{1}{\kd}-1} \geq \frac{C_0}{C^{2}_{p}}\frac{\bmi}{\omi \oma}  \frac{\omi}{\oma}=1,
}
where in the last inequality we applied $\frac{1}{\kd}-1<  1$. Hence, (\ref{C0InEq0}) is proved for $\kd>\frac{1}{2}$.

Next, applying  (\ref{C0InEq0}) in  (\ref{tmp::1}) we deduce that
\eqns{
\ddt \n{ \nj{b} + \ndk{ v } }
+ \min \left\{1,\frac{1}{C_0} \right\} \omitt \n{ \nj{b} + \ndk{ v} }
\leq 0.
}
After integration from $0$ to $t$ we get
\eqns{
\ln \n{\frac{ \nj{b} + \ndk{ v } }{ \nj{b_0} + \ndk{ v_0 }}}
\le
- \frac{1}{\kd} \min \left\{1,\frac{1}{C_0} \right\} \ln \n{1+\kd \omi t},
}
which gives  (\ref{b-l1-est2}).

{\bf f) } The estimate  (\ref{b-l1-est3}) is a direct consequence of (\ref{b-l1-norm-ineq}).

\end{proof}

\subsection{Higher order estimates}
In this section we will obtain estimates for $\nd{\lap v(t)}$, $\nd{\lap \om(t)}$ and $\nd{\lap b(t)}$. Having these estimates and results of the previous section we will be able to control the $H^{2}$ norm. From (\ref{gbh})-(\ref{ccgno}) we get
\eqq{(v_{,t}, \lapk w) -  (v\otimes v,\nabla \lapk w)  +     \n{ \frac{b}{\om}  D(v), D(\lapk w)} = 0,}{gbha}
\eqq{ (\om_{,t}, \lapk z) - (\om v, \nabla \lapk z  ) + \n{ \frac{b}{\om} \nabla \om, \nabla \lapk z  } = - \kd (\om^{2}, \lapk z ),  }{cbgha}
\eqq{(b_{,t}, \lapk q) - (b v, \nab \lapk q   )   +\n{  \frac{b}{\om}  \nabla b, \nabla \lapk q  }  = - (b \om , \lapk q) + \left( \frac{b}{\om} \bk{D(v)}, \lapk q\right) ,  }{ccgnoa}
for a.a. $t\in (0, \Ts)$,  where the test functions are such that $\lapk w \in \Vkdj$, $\lapk z \in \Vj$ and $ \lapk q \in \Vj$. If we integrate by parts and use the condition $\divv v =0$, then  we obtain
\eqq{\langle \lap v_{,t}, \lap w\rangle  -  (\lap \n{ v\otimes v},\nabla \lap w)  +     \n{ \lap \n{ \frac{b}{\om}  D(v)}, D(\lap w)} = 0,}{gbhb}
\eqnsl{ \langle\lap \om_{,t}, \lap z\rangle  - (v \nabk \om ,  \nabla \lap z  )- (\nab \om \nab v ,\nab \lap z) + &   \n{ \lap \n{\frac{b}{\om}   \nabla \om}, \nabla \lap z  } \\ &  = - \kd (\lap \n{\om^{2}}, \lap z ),  }{cbghb}
\eqnsl{ \langle \lap b_{,t}, \lap q\rangle   - (v \nabk b, \nab \lap q  &  )- (\nab b \nab \om , \nab \lap q)   +\n{  \lap \n{ \frac{b}{\om}  \nabla b}, \nabla \lap q  } \\
&  = - (\lap \n{b \om} , \lap q) - \left( \nab \n{\frac{b}{\om} \bk{D(v)}}, \lap \nab q\right) , }{ccgnob}
for a.a. $t\in (0, \Ts)$,  where $\langle \cdot , \cdot \rangle $ denotes duality pairing between $\V^{1}(\Om)$ and $(\V^{1})^{*}$. The density argument and regularity of $(v, \om, b)$   allow us to test the system (\ref{gbhb})-(\ref{ccgnob}) by solution thus, we obtain
\eqq{\jd \ddt \ndk{ \lap v}  -  (\lap \n{ v\otimes v},\nabla \lap v)  +     \n{ \lap \n{ \frac{b}{\om}  D(v)}, D(\lap v)} = 0, }{gbhc}
\eqnsl{
\jd \ddt \ndk{\lap \om}   - (v & \nabk  \om,  \nab \lap \om   ) - (\nab \om \nab v , \nab \lap \om) \\
&   + \n{ \lap \n{\frac{b}{\om} \nabla \om}, \nabla \lap \om  } = - \kd (\lap \n{\om^{2}}, \lap \om ),  }{cbghc}
\eqnsl{
\jd \ddt \ndk{  \lap b}  -  (v  \nabk b, & \nab \lap b   )- (\nab b  \nab \om , \nab \lap b)   +\n{  \lap \n{ \frac{b}{\om}  \nabla b}, \nabla \lap b  } \\
& = - (\lap \n{b \om} , \lap b) - \left( \nab \n{\frac{b}{\om} \bk{D(v)}}, \nab \lap  b\right) }{ccgnoc}
for a.a. $t\in (0, \Ts)$. In the above equations some terms are similar and can be treated in the same way. To simplify further calculations let us analyse these terms first. One of them  has the following form
\[
\n{ \lap \n{\frac{b}{\om} \nabla f}, \nabla \lap f }.
\]
In this case we may write
\eqnsl{
& \n{ \lap \n{\frac{b}{\om} \nab f}, \nab \lap f }  =\n{ \frac{b}{\om} \nal f, \nal f } + 2 \n{ \nabk f \cdot \nab \n{\frac{b}{\om}} , \nal f } \\
 + &
\n{ \lap \n{\frac{b}{\om}} \nab f, \nal f }=   \n{ \frac{b}{\om} \nal f, \nal f } +  2 \n{\frac{1}{\om} \nabk f \cdot \nab b , \nal f } \\
-&  2 \n{  \frac{b}{\om^2} \nabk f \cdot \nab \om , \nal f } +\n{ \frac{\lap b}{\om} \nab f, \nal f }
- 2 \n{ \frac{(\nab b\cdot \nab \om )}{\om^2} \nab f, \nal f } \\
- &   \n{ \frac{b}{\om^2} \lap \om \nab f, \nal f }+ 2\n{ \frac{b}{\om^3} \bk{\nab \om} \nab f, \nal f }.
}{pom1}
On the right-hand side we can control the sign only of the first term hence, to simplify the future calculations we define $W(f)$ using the last six expressions, i.e.
\eqq{
\n{ \lap \n{\frac{b}{\om} \nab f}, \nab \lap f } =
\n{ \frac{b}{\om} \nal f, \nal f }+ W(f).
}{diffTm-lapk}
Similarly we define $\wt{W}(v)$
\eqq{
\n{ \lap \n{\frac{b}{\om} D(v) }, D( \lap v) } =
\n{ \frac{b}{\om} D(\lap v) ,  D (\lap v) }+ \wt{W}(v).
}{diffTm-lapkv}
Using this notation the  system (\ref{gbhc})-(\ref{ccgnoc}) may be written in the  following way
\eqnsl{
\jd \ddt \ndk{\lap v}  +  \n{ \frac{b}{\om}  D(\lap v), D(\lap v)} = (\lap \n{v \otimes v},\nab \lap v) - \wt{W}(v)}{v-lapkTesteda}
\eqnsl{
\jd \ddt \ndk{\lap \om} + \n{ \frac{b}{\om} \nal \om, \nal \om  } = - \kd (\lap (\om^{2}), \lap \om ) + (v \nabk \om , \nabla \lap \om  ) \\ +(\nab \om \nab v, \nal \om) - W(\om),  }{om-lapkTesteda}
\eqnsl{
\jd \ddt \ndk{\lap b} + \n{  \frac{b}{\om}  \nal b, \nal b  } = (v\nabk b, \nal b )+ (\nab b \nab v , \nal b)
 \\ -(\lap \n{b \om},  \lap b) - \left(\nab \n{ \frac{b}{\om} \bk{D(v)}}, \nab \lap b\right) - W(b).
}{b-lapkTesteda}
We recall that by applying (\ref{mnitDef}) and (\ref{inf-EST-b-omega}) we get the bound from below
\eqq{
\mnit \leq \frac{b}{\om}
}{below}
thus, from (\ref{v-lapkTesteda}) we obtain
\eqnsl{
\jd \ddt \ndk{\lap v} + \mnit \ndk{D( \lap v)} & \le 2 ( \lap v \otimes v , \nal v ) +  2( \nab v \otimes \nab v , \nal v ) - \wt{W}(v) .
}{v-lapk-tested-2}
To estimate the right-hand side we use the H\"older inequality and we get
\[
2( \lap v \otimes v , \nal v ) +  2( \nab v \otimes \nab v , \nal v ) \leq 2\nt{v} \ns{ \lap v } \nd{ \nal v }
+ \nck{ \nab v } \nd{ \nal v }.
\]
Then, after applying Sobolev inequalities  and Gagliardo-Nierenberg inequality (\ref{gag50}) we get
\[
2( \lap v \otimes v , \nal v ) +  2( \nab v \otimes \nab v , \nal v ) \leq C( \nt{ v} +\nd{ \nab v })\ndk{ \nabt v},
\]
where $C$ depends only on $\Om$. If we use the interpolating inequality
\[
\nt{ v} \leq C \nd{v}^{\jd} \nd{ \nab v }^{\jd},
\]
then we obtain
\eqq{
\jd  \ddt \ndk{\lap v } + \mnit  \ndk{ D(\lap v) } \leq  C
\n{\nd{ v }^{\frac{1}{2}} \nd{ \nab v }^{\frac{1}{2}} + \nd{ \nab v } }
 \ndk{ \nabt  v}  - \wt{W}(v).
}{est_v_lap_GLOB}
Now we focus our attention on the equation (\ref{om-lapkTesteda}). After applying  (\ref{below}) we get
\eqn{
\jd \ddt \ndk{\lap \om}  + \mnit \ndk{\nab \lap \om}  \le (v\nabk \om, \nal \om)+ & (\nab \om \nab v, \nal \om)\\
& -2 \kd (\bk{\nab \om}, \lap \om )
 - W(\om),  }
where we used the nonnegativity of $ 2\kd \n{ \om \lap \om , \lap  \om   }$. By H\"older inequality we have
\eqn{
\jd \ddt \ndk{ \lap \om} + \mnit \ndk{ \nal \om}  \le  \nt{v} \ns{\nabk \om}&  \nd{\nal \om} +
2\nc{\nab \om} \nc{ \nab v} \nd{\nal \om}
 \\
&
+2 \kd \norm{\nab \om}_{\frac{6}{5}} \nif{\nab \om} \ns{\lap \om}
- W(\om).
}
After applying the estimate (\ref{gag_inf_lap}) to the term $\nif{\nab \om}$ and (\ref{est-zero-3}) to term $ \nt{v}$ we obtain
\eqn{
\jd \ddt \ndk{ \lap \om} + \mnit \ndk{ \nal \om}  \le
  C \Big(
\nd{v}^{\jd}  \nd{\nab v}^{\jd}
& + \kd  \norm{\nab \om}_{\frac{6}{5}}
 \Big) \ndk{\nabt \om} & \\ &
+  2\nc{\nab \om} \nc{ \nab v} \nd{\nal \om} - W(\om) .
}
If we use the inequality (\ref{est-nab-4-kw}) then we get
\eqn{
\jd \ddt \ndk{ \lap \om} + \mnit \ndk{ \nal \om} & \le C \Big(
\nd{v}^{\jd}\nd{\nab v}^{\jd}
+ \kd \norm{\nab \om}_{\frac{6}{5}} \Big) \ndk{\nabt \om} \\
& + C\norm{\nab \om}_{\td}^{\jd} \norm{\nab v}_{\td}^{\jd} \nd{\nabt \om}^{\frac{3}{2}} \nd{\nabt v}^{\frac{1}{2}}  - W(\om) ,
}
where $C$ depends only on $\Om$. So finally, after applying the Young inequality with exponents  ($\frac{4}{3}$, 4) we obtain
\[
\jd \ddt \ndk{ \lap \om} + \mnit \ndk{ \nal \om}
\]
\eqnsl{
\le C \Big(
\nd{v}^{\jd}\nd{\nab v}^{\jd}
+ \kd \norm{\nab \om}_{\frac{6}{5}}
 + \norm{\nab \om}_{\td}^{\jd} \norm{\nab v}_{\td}^{\jd}
 \Big) \n{ \ndk{\nabt \om} + \ndk{\nabt v} }  - W(\om).
}{est_om_lap_GLOB}
Now, let us turn our attention to equation (\ref{b-lapkTesteda}). We integrate by parts
\eqn{
- (\lap \n{b \om}, \lap b )
& =
- ( \om \lap b, \lap b )
- 2(\nab \om \nab b, \lap b )
- ( b   \lap \om, \lap b ),
}
\eqn{
\n{ \nab \n{\frac{b}{\om} |D(v)|^2}, \nal b}  & \\
=
\n{\frac{\nab b}{\om} |D(v)|^2 , \nal b }
-\n{\frac{b\nab \om }{\om^2}  |D(v)|^2 , \nal b } &
+ 2 \n{\frac{b}{\om}  D(v) \nab D(v) , \nal b}.
}
Using the above calculations we may write  (\ref{b-lapkTesteda}) in the following form
\[
\jd \ddt \ndk{\lap b} + \n{ \frac{b}{\om} \nal b, \nal b }
 =(v \nabk b , \nal b ) +
(\nab b \nab v, \nal b )
 - ( \om \lap b, \lap b )
\]
\[
- 2(\nab \om \nab b, \lap b ) - ( b \lap \om, \lap b )
 -\n{\frac{\nab b}{\om} |D(v)|^2 , \nal b }
+\n{\frac{b\nab \om }{\om^2}  |D(v)|^2 , \nal b }
\]
\[
- 2 \n{\frac{b}{\om}  D(v) \nab D(v) , \nal b}
 - W(b).
\]
The third term on the right-hand side is non-positive hence, if we  use (\ref{below}), then  we get
\[
\jd \ddt \ndk{\lap b} + \mnit \ndk{ \nal b }
 \le
\]
\[
(\nab b \nab v, \nal b ) + (v \nabk b , \nal b )
 - 2(\nab \om \nab b, \lap b ) - ( b \lap \om, \lap b )
 -\n{\frac{\nab b}{\om} |D(v)|^2 , \nal b }
\]
\eqq{
+\n{\frac{b\nab \om }{\om^2}  |D(v)|^2 , \nal b }
- 2 \n{\frac{b}{\om}  D(v) \nab D(v) , \nal b}
 - W(b).
}{b-lapk-tested-2}
From  H\"older inequality we obtain
\[
\jd \ddt \ndk{ \lap b} + \mnit \ndk{ \nal b}  \le
  \nc{\nab b } \nc{\nab v} \nd{ \nal b}
+ \nt{v} \ns{\nabk b} \nd{ \nal b}
\]
\[
+2 \norm{\nab \om}_{\frac{6}{5}} \nif{\nab b} \ns{\lap b}
+ \norm{b}_{\frac{3}{2}} \ns{\lap \om} \ns{\lap b} + \nif{\re{\om}} \ns{ \nab b} \nsk{ D (v)} \nd{\nal b}
\]
\[
+ \nif{\re{\om}}^2 \nif{b} \ns{\nab \om} \nsk{ D (v)} \nd{\nal b}
\]
\[
+ 2\nif{\re{\om}}  \nif{b} \ns{ \nab D (v)} \nt{ D (v)} \nd{\nal b} -W(b).
\]
Now, we estimate the right-hand side by applying Gagliardo-Nirenberg inequalities
\[
\m{ \hd by } (\ref{est-nab-4-kw}): \hspace{1.2cm} \hd  \hd \nc{\nab b } \nc{\nab v} \nd{ \nal b} \leq c \norm{ \nab b }^\jd_\td \norm{ \nab v }^\jd_\td \nd{ \nabt b }^\jd \nd{ \nabt v}^\jd,
\]
\[
\m{ \hd by } (\ref{est-zero-3}), (\ref{Sobolev}): \hspace{2.4cm} \hd  \hd \nt{v} \ns{\nabk b} \nd{ \nal b} \leq c \nd{\nab v}^\jd \nd{v}^\jd \ndk{\nabt b},
\]
\[
\m{ \hd by } (\ref{Sobolev}), (\ref{gag_inf_lap}): \hspace{3.0cm} \hd \hd
\norm{\nab \om}_{\frac{6}{5}} \nif{\nab b} \ns{\lap b} \leq c \norm{\nab \om}_{\frac{6}{5}} \ndk{\nabt b} ,
\]
\[
\m{ \hd by } (\ref{32Est}), (\ref{Sobolev}): \hd \hd
\norm{b}_{\frac{3}{2}} \ns{\lap \om} \ns{\lap b} \leq c (\norm{\nab b}^\jd_\td \nj{b}^\jd + \nj{b}) \nd{ \nabt \om } \nd{\nabt b},
\]
\[
\m{\hd by  }  (\ref{inf-EST-b-omega}), (\ref{Sobolev}), (\ref{est-nab-6}): \hspace{0.5cm} \hd \hd \nif{\re{\om}} \ns{ \nab b} \nsk{ D (v)}
 \nd{\nal b} \hspace{3.5cm}
\]
\[
\hspace{6.2cm} \leq c (\omt)^{-1} \nd{\nabk b} \nd{ \nab v} \nd{ \nabt v} \nd{\nabt b},
\]
\[
\m{\hd by  }  (\ref{inf-EST-b-omega}), (\ref{gag_inf_lap_j}), (\ref{Sobolev}), (\ref{est-nab-6}): \hd \hd \nif{\re{\om}}^2 \nif{b} \ns{\nab \om} \nsk{ D (v)} \nd{\nal b} \hspace{2cm}
\]
\[
\hspace{3.7cm} \leq c (\omt)^{-2} (\nd{ \nabk b} + \nj{b}) \nd{ \nabk \om} \nd{ \nab v} \nd{ \nabt v} \nd{ \nabt b},
\]
\[
\m{\hd by  }  (\ref{inf-EST-b-omega}), (\ref{gag_inf_lap_j}), (\ref{Sobolev}), (\ref{est-zero-3}) : \hd \hd \nif{\re{\om}}  \nif{b} \ns{ \nab D (v)} \nt{ D (v)} \nd{\nal b} \hspace{1.3cm}
\]
\[
\hspace{3cm} \leq c (\omt)^{-1} (\nd{ \nabk b} + \nj{b}) \nd{ \nabt v} \nd{\nab v }^{\jd} \nd{ \nabk v}^{\jd } \nd{ \nabt b},
\]
where $c$ depends only on $\Om$. Thus, if we apply Young inequality to separate the norms of the third order derivatives, then we obtain
\eqnsl{
\jd \ddt \ndk{ \lap b} + \mnit \ndk{ \nal b}  \le c \Big(
\norm{\nab b}_{\td}^{\jd} \norm{\nab v}_{\td}^{\jd}
+ \nd{\nab v}^{\jd}\nd{ v}^{\jd}
+ \norm{\nab \om}_{\frac{6}{5}}
+ \norm{\nab b}^\jd_\td \nj{b}^\jd  \\
+ \nj{b} + (\omt)^{-1} \nd{ \nabk b} \nd{ \nab  v}
+ (\omt)^{-2} \n{\nd{\nabk b} + \nj{b}} \nd{\nabk \om} \nd{ \nab  v} \\
+ (\omt)^{-1} \n{\nd{\nabk b} + \nj{b}} \nd{ \nab v}^{\jd} \nd{ \nabk v}^{\jd}
 \Big)  \cdot \Big( \ndk{\nabt v} + \ndk{\nabt \om} + \ndk{\nabt b} \Big)- W(b),
}{b-lapk-tested-3}
where $c$ depends only on $\Om$. We note that after integration by parts we get $\nd{\nabk f} =\nd{ \lap f}$ for $f\in \mathcal{V}^{1}$ and  $2 \ndk{ D(\lap v)}= \ndk{ \nabt v }$ (see (47) \cite{KoKu}) hence, if we  sum the inequalities (\ref{est_v_lap_GLOB}), (\ref{est_om_lap_GLOB}) and (\ref{b-lapk-tested-3}), then we obtain
\eqn{
\jd  \ddt  \n{ \ndk{\lap v } + \ndk{\lap \om } + \ndk{\lap b } }
+ \mnit \n{ \ndk{\nal v} + \ndk{\nal \om} + \ndk{\nal b} } \\
 \le C
\Big(
\nd{ v }^\jd \nd{ \nab v }^\jd
+ \nd{\nab v}
+  \norm{\nab \om}_{\frac{6}{5}}
+ \norm{\nab \om}_{\td}^{\jd} \norm{\nab v}_{\td}^{\jd}
 + \norm{\nab b}_{\td}^{\jd} \norm{\nab v}_{\td}^{\jd} \\
+ \norm{\nab b}^\jd_\td \nj{b}^\jd + \nj{b}
+ (\omt)^{-1}  \nd{ \nabk b} \nd{ \nab v}
+ (\omt)^{-2} \nd{\nabk b} \nd{\nabk \om} \nd{ \nab v} \\
+ \frac{\nj{b}}{(\omt)^2}  \nd{\nabk \om} \nd{ \nab v}
+ (\omt)^{-1} \nd{\nabk b} \nd{ \nab v}^{\jd} \nd{ \nabk v}^{\jd}
+ \frac{\nj{b}}{\omt}  \nd{ \nab v}^{\jd} \nd{ \nabk v}^{\jd}
\Big) \\
 \cdot \Big( \ndk{\nal v} + \ndk{\nal \om} + \ndk{\nal b} \Big)  - \wt{W}(v) - W(\om) - W(b),
}
where $C$ depends only on $\kd$ and $ \Omega $. Before we estimate the last three terms we will introduce the following notation
\eqnsl{
& X_0(t) := \ndk{v(t)} + \nj{b(t)}^2, \\
&
X_1(t) := \ndk{\nab v(t)} + \ndk{\nab \om(t)} + \ndk{\nab b(t)}, \\
&
X_2(t) := \ndk{\lap v(t)} + \ndk{\lap \om(t)} + \ndk{\lap b(t)}, \\
&
X_3(t) := \ndk{\nal v(t)} + \ndk{\nal \om(t)} + \ndk{\nal b(t)}.
}{X_k_def}
After using the H\"older inequality we obtain
\[
\jd  \ddt  X_2 + \mnit X_3  \le C
\Big(
 X_0^\jc X_1^\jc
+ X_1^\jd+ \nj{b}
+ (\omt)^{-1} X_1^\jd X_2^\jd
+ (\omt)^{-2} X_1^\jd X_2
\]
\eqq{
+ \frac{\nj{b}}{(\omt)^2}  X_1^\jd X_2^\jd
+ (\omt)^{-1} X_1^\jc X_2^\tc
+ \frac{\nj{b}}{\omt}  X_1^\jc X_2^\jc
 \Big) \cdot X_3  - \wt{W}(v) - W(\om) - W(b),
}{est4Sum}
where $C$ depends only on $\kd$ and $ \Omega $.
Now, we need to estimate terms $ \wt{W}(v) $, $ W(\om) $, $ W(b) $, which were defined by (\ref{pom1})-(\ref{diffTm-lapkv}). In each case the estimates are similar thus, we consider $W(f)$ for general $f\in \mathcal{V}^{3}$. In this case we have
\eqn{
|W(f)|  \le  2 & \nif{\re{\om}} \nt{ \nab b} \ns{ \nabk f} \nd{\nal f}
+ 2 \nif{\re{\om}}^2 \nif{b} \nt{\nab \om} \ns{\nabk f} \nd{\nal f} \\
 + &  \nif{\re{\om}} \ns{ \lap b } \nt{\nab f } \nd{ \nab \lap f}
+ 2 \nif{\re{\om}}^2 \ns{ \nab b} \ns{ \nab \om} \ns{ \nab f} \nd { \nal f } \\
+ &
 \nif{\re{\om}}^2 \nif{b} \ns{ \lap \om} \nt{\nab f} \nd{\nal f}
+ 2 \nif{\re{\om}}^3 \nif{b} \ns{ \nab \om}^2 \ns{ \nab f} \nd{\nal f}. \\
}
As earlier, we use (\ref{inf-EST-b-omega}) and (\ref{est-zero-3})-(\ref{32Est}) and we have
\[
|W(f)|
\]
\eqns{
  \le \frac{c}{\omt} \Big(
  \nd{ \nab b}^\jd \nd{ \lap b}^\jd \nd{\nabt f}
+ (\omt)^{-1} \n{\nd{\lap b} + \nj{b} } \nd{ \nab \om}^\jd \nd{ \lap \om}^\jd & \nd{\nabt f} \\
+  \nd{ \nab f }^\jd \nd{ \lap f }^\jd \nd{ \nal b}
+ (\omt)^{-1} \nd{\nab b}^\jd \nd{\nab \om}^\jd \nd{ \lap f} \nd{\nal \om}^\jd & \nd{\nal b}^\jd \\
+ (\omt)^{-1} \n{ \nd{\lap b} + \nj{b} }  \nd{ \nab f}^\jd \nd{ \lap f}^\jd & \nd{ \nal \om} \\
+ (\omt)^{-2} \n{ \nd{\lap b} + \nj{b} } \nd{\nab \om} \nd{ \lap f} \nd{\nal \om}
  \Big) & \nd { \nal f }, \\
}
where $c$ depends only on $\Om$. We obtain an analogous estimate for $\wt{W}(v)$. Then, if we use the notation (\ref{X_k_def}), then  we obtain
\eqn{
|\wt{W}(v)|+|W(\om)|+|W(b)| \\
  \le \frac{c}{\omt} \Big(
   X_1^\jc X_2^\jc
+ (\omt)^{-1} X_1^\jc X_2^\tc
+ \frac{\nj{b}}{\omt} X_1^\jc X_2^\jc
+   X_1^\jc X_2^\jc
+ (\omt)^{-1} X_1^\jd X_2^\jd \\
+ (\omt)^{-1} X_1^\jc X_2^\tc
+ \frac{\nj{b}}{\omt} X_1^\jc X_2^\jc
+ \re{(\omt)^2} X_1^\jd X_2
+ \frac{\nj{b}}{(\omt)^2} X_1^\jd X_2^\jd
\Big) \cdot X_3,
}
where $c$ is as earlier. We simplify  further
\[
|\wt{W}(v)|+|W(\om)|+|W(b)| \le \frac{c}{(\omt)^{2}} \cdot   \hspace{6cm}
\]
\eqq{
\cdot   \left(
  \n{\omt+\nj{b}} X_1^\jc X_2^\jc
+ \n{ 1 + \frac{\nj{b}}{\omt}} X_1^\jd X_2^\jd
+  X_1^\jc X_2^\tc
+ (\omt)^{-1} X_1^\jd X_2
\right) \cdot X_3
}{W_est}
and $c$  depends only on $\Om$.  Using this estimate in (\ref{est4Sum}) we get
\eqnsl{
\jd  \ddt  X_2 + \mnit X_3  \le C
\Big(
X_0^\jc X_1^\jc
+ X_1^\jd +\norm{b}_{1}
+  \n{\re{\omt} + \frac{\nj{b}}{\omt} + \frac{\nj{b}}{(\omt)^2} } X_1^\jc X_2^\jc \\
+ \n{\re{\omt} + \re{(\omt)^2} + \frac{\nj{b}}{(\omt)^2} + \frac{\nj{b}}{(\omt)^3}} X_1^\jd X_2^\jd
+ \n{\re{\omt} + \re{(\omt)^2}} X_1^\jc X_2^\tc \\
+ \n{\re{(\omt)^2} + \re{(\omt)^3}} X_1^\jd X_2
 \Big) \cdot X_3,
}{est4Sum_v1}
where $C=C(\Omega, \kd)$. After applying the  Poincar\'e inequality we get $ X_1 \le C^{2}_{p} X_2 $ thus, we may simplify further
\eqnsl{
\jd  \ddt  X_2 + \mnit X_3  \le C
\Big(
X_0^\jc X_2^\jc + \norm{b}_{1}
+  \n{1 + \re{\omt} + \frac{\nj{b}}{\omt} + \frac{\nj{b}}{(\omt)^2} } X_2^\jd \\
+ \n{\re{\omt} + \re{(\omt)^2} + \frac{\nj{b}}{(\omt)^2} + \frac{\nj{b}}{(\omt)^3}} X_2
+ \n{ \re{(\omt)^2} + \re{(\omt)^3}} X_2^\frac{3}{2}  \Big) \cdot X_3.
}{est4Sum_v2}
By (\ref{mnitDef}) and (\ref{b-l1-est}) we have $\nj{ b(t) }\leq \bma(t)$ hence, using (\ref{def_A}), (\ref{noa}) and  (\ref{X_k_def}) we get
\eqns{
X_0^\jc(t) & \leq \n{ \ndk{
v_{0}}
\exp \n{ - \frac{2\bmi \n{\n{1 + \kd \oma t }^{2 - \frac{1}{\kd}}-1} }{C^{2}_{p} \oma^{2} \n{2 \kd - 1 }}   }
+ \bma^{2}(t)}^{\jc} \\
& \equiv A(t)
}
and we obtain
\[
X_0^\jc X_2^\jc + \norm{b}_{1} \leq
A(t) X_2^{\jc}+ \bmmt.
\]
Applying this inequality in (\ref{est4Sum_v2}) we get
\eqnsl{
\ddt & X_2 + 2 \mnit X_3  \le C_{\Om,\kd}
\Big( \bmmt +
A(t) X_2^\jc
+ B(t) X_2^\jd
+ C(t) X_2
+ D(t) X_2^\frac{3}{2}  \Big) \cdot X_3,
}{est4Sum_v3}
where $C_{\Om,\kd}$ depends only on $\Om$, $\kd$ and  we used the notation (\ref{def_B})-(\ref{def_D}). We denote
\eqnsl{
Z(t) & =
\Big( \bmmt +
A(t) X_2^\jc
+ B(t) X_2^\jd
+ C(t) X_2
+ D(t) X_2^\frac{3}{2}
\Big).
}{def_Z}
Thus,   the inequality (\ref{est4Sum_v3}) may be written in the following form
\eqn{
\ddt X_2(t) + \n{\mnit - \comkd Z(t)}  X_3(t) & \le -\mnit X_3(t).
}
By Poincar\'e inequality we get
\eqnsl{
\ddt X_2(t) + \n{\mnit - \comkd Z(t)}  X_3(t) & \le - \frac{\mnit}{C^{2}_{p}} X_2(t).
}{key1}
By definition (\ref{Y_0t}) and (\ref{X_k_def}) we have $Y_{2}(0)=X_{2}(0)$ hence, using (\ref{def_Z0})  and (\ref{def_Z}) we get $Z_{0}(0)=Z(0)$. Next, by assumption (\ref{GLOB_ADD}) we have
\[
\frac{\bmi}{\oma}- \comkd Z_{0}(0)>0
\]
thus, we have
\[
\frac{\bmi}{\oma}- \comkd Z(0)>0.
\]
We note that $(v, \om ,b)\in L^{2}([0,\Ts);H^{3}(\Om))$ and $(v_{,t}, \om_{,t} ,b_{,t})\in L^{2}([0,\Ts);H^{1}(\Om))$ hence, we have $X_{2}\in C([0,\Ts))$. Therefore, there are two possibilities:
\[
\forall t \in [0,T^{*}) \hd \hd \mnit- \comkd Z(t)>0 \hd \m{ or } \hd  \exists t^{*}\in (0,T^{*}) \hd \hd \mu^{t^*}_{\min} - Z(t^*) = 0.
\]
In the first case, the inequality (\ref{key1}) gives a uniform estimate
\eqq{
\ndk{\lap v (t) }+ \ndk{\lap \om  (t) }+ \ndk{\lap b (t) } \leq \ndk{\lap v_{0}  }+ \ndk{\lap \om_{0}  }+ \ndk{\lap b_{0}  } \hd \m{ for } \hd t\in [0,\Ts).
}{contra}
By (\ref{noa})-(\ref{b-l1-est}) we have
\[
\nd{ v(t)} \leq \nd{v_{0}}, \hd  \nd{ \om(t)} \leq \nd{ \om_{0}},
\]
\[
 \nd{ b(t)} \leq c(\nd{ \nabk b(t)} + \nj{ b(t)}) \leq c(\nd{ \nabk b(t)} + \nj{ b_{0}} + \jd \ndk{ v_{0}}  )
\]
for $ t\in [0,\Ts)$,  where $c=c(\Om)$. These estimates together with  (\ref{contra}) give
\eqq{ \nsodk{v (t) }+ \nsodk{ \om  (t) }+ \nsodk{ b (t) } \leq c\n{\nsodk{ v_{0}  }+ \nsodk{ \om_{0}  }+ \nsodk{ b_{0}  }  }
 }{contrb}
for $ t\in [0,\Ts)$, where $c$ depends only on $\Om$. We denote the right-hand side of (\ref{contrb}) by $\delta$. We set $K=\{(\omt, \ommt, \bmt) : \hd t\in [0,\Ts] \}$. Then $K$ is compact subset of $\{(a,b,c): \hd 0<a\leq b, \hd 0<c \}$ and by Theorem~\ref{LOCALNE_TH} there exists $\tskd$ such that the problem (\ref{a})-(\ref{ddod}) with initial condition $(v(t), \om(t), b(t))$ can be extended to the interval $[t,t+\tskd)$, where $t $ is arbitrary in $[0,\Ts)$. For $t>\Ts-\tskd$ we obtain the contradiction with definition of $\Ts$ (see (\ref{defTgw})).

In the second case, using the continuity of $[0,T^{*})\ni t \mapsto \mnit- \comkd Z(t) $ we may assume that $t^{*}\in (0,T^{*})$ is the first point with this property, i.e. \m{$\mnit- \comkd Z(t)>0$} for $t\in [0,t^{*})$ and $\mu^{t^*}_{\min} - Z(t^*) = 0$. Then, from (\ref{key1}) we get
\eqn{
\ddt  X_2(t) \le -\frac{1}{C^{2}_{p}}\mnit X_2(t) \hd \m{ for } \hd t\in (0,t^{*}).
}
Using (\ref{mnitDef}) we may write
\eqn{
\ddt  X_2(t) \le - \frac{1}{C^{2}_{p}} \frac{\bmi}{\oma}
\n{1 + \kd \oma t}^{1 - 1/\kd} X_2(t) \hd \m{ for } \hd t\in (0,t^{*}).
}
Thus, after multiplying by appropriate exponential function we obtain the bound
\eqn{
X_2(t) & \le
X_2(0) \exp \n{-\frac{1}{C^{2}_{p}}\frac{\bmi }{(2\kd - 1)\oma^2} \n{ \n{1 + \kd \oma t}^{2 - 1/\kd} - 1}}  \hd \m{ for } \hd t\in (0,t^{*}).
}
By definition (\ref{Y_0t}), the above inequality means $X_{2}(t)\leq Y_{2}(t)$ for $t\in [0,t^{*})$ hence, we get $X_{2}(t^{*})\leq Y_{2}(t^{*})$.  If we use the definition (\ref{def_Z0}) and (\ref{def_Z}), then we deduce that $Z(t^*) \le Z_0(t^*)$ and then
\eqn{
0=\mu^{t^*}_{\min} - \comkd Z(t^*) \geq \mu^{t^*}_{\min} - \comkd Z_0(t^*) > 0
}
and we get a contradiction with the assumption (\ref{GLOB_ADD}). Thus, we obtain that $\Ts \geq T$ and the theorem~\ref{TW_GLOBAL} is proved.

It remains to prove Corollary~\ref{coro_glob}.

\begin{proof}[Proof of Corollary~\ref{coro_glob}]
We shall show that the condition (\ref{GLOB_ADD}) is satisfied for $T=\infty$.
Firstly, for $\kd \geq 1$ we note that from (\ref{Z1}) we obtain
\[
\frac{\bmi}{\oma} >2 \comkd \n{ \| b_{0}\|_{1} + \jd \ndk{ v_{0}} \n{1 + I_\infty \n{\kd, \frac{\omi}{\oma}, \frac{\bmi}{(\oma)^2}} } } \n{1 + \kd \omi t}^{-1}
\]
for $t\geq 0$ hence, after multiplying both sides by $\n{1 + \kd \omi t}^{1-\frac{1}{\kd}}$ we get
\eqq{\mnit \geq \frac{\bmi}{\oma} \n{1 + \kd \omi t}^{1-\frac{1}{\kd}} > 2 \comkd \bmmt .}{W1}
For $\kd \in \n{\jd, 1}$ we note that from (\ref{Z1.5}) we obtain
\eqns{
\frac{\bmi}{\oma} & \n{1 + \kd \oma t} \\
& >
2 \comkd \n{\frac{\oma}{\omi}}^\frac{1}{\kd} \n{ \| b_{0}\|_{1} + \jd \ndk{ v_{0}} \n{1 + I_\infty \n{\kd, \frac{\omi}{\oma}, \frac{\bmi}{(\oma)^2}} } }
}
for $t\geq 0$ hence, after multiplying both sides by $\n{1 + \kd \oma t}^{-\frac{1}{\kd}}$ we get
\[
\mnit >
2 \comkd \bmmt
 \n{\frac{\oma}{\omi}\cdot  \frac{1 + \kd \omi t}{1 + \kd \oma t}}^\frac{1}{\kd}.
\]
We note that the function  $\frac{1 + \kd \omi t}{1 + \kd \oma t} $ is decreasing and strictly greater than $\frac{\omi}{\oma}$ so, we have
\eqq{
\mnit >
2 \comkd \bmmt .}{W1.5}

Next, we shall show that $a_{0}$ is finite. Recall that $\kd > \jd$ and then  by (\ref{bMinMaxt}), (\ref{def_A})  we deduce that \m{$\jkt^{\frac{1}{\kd}-1}A(t)$} decays at infinity as $\jkt^{\frac{1}{2\kd}-1}$. Thus, the expression \m{$\jkt^{\frac{1}{\kd}-1}A(t)$} is uniformly bounded on $[0,\infty)$. Further, the remaining terms in definition $a_{0}$ can be estimated by expressions of the form $\jkt^{\alpha}Y_{2}^{\beta}(t)$, where $\alpha\leq 3 $ and $\beta>0$. We recall that the function  $Y_{2}(t)$ decays exponentially hence, $a_{0}$ is finite.

Finally, by (\ref{Z2}) we get $\frac{\bmi}{\oma}>a_{0} Y_{2}^{\frac{1}{4}}(t)$ for $t\in [0,\infty)$ thus, using the definition of $a_{0}$ we obtain
\[
\frac{\bmi}{\oma}>2\comkd \jkt^{\frac{1}{\kd}-1} \left(A(t)Y_{2}^{\frac{1}{4}}(t)+ B(t)Y_{2}^{\frac{1}{2}} + C(t)Y_{2}(t)+ D(t)Y_{2}^{\frac{3}{2}}(t)  \right)
\]
for $t\in [0,\infty)$ hence, we get
\eqq{
\mnit>2\comkd  \left(A(t)Y_{2}^{\frac{1}{4}}(t)+ B(t)Y_{2}^{\frac{1}{2}} + C(t)Y_{2}(t)+ D(t)Y_{2}^{\frac{3}{2}}(t)  \right).
}{W2}
If we sum (\ref{W1}) or (\ref{W1.5}) and (\ref{W2}) then, by definition (\ref{def_Z0}) we get $2\mnit > 2\comkd Z_{0}(t) $ hence,
the condition (\ref{GLOB_ADD}) holds for $T=\infty.$
\end{proof}

{\bf Acknowledgements } The authors would like to thank the anonymous referee for valuable remarks, which
significantly improve the paper.

\section{Appendix}
In this subsection we collect the special cases of Gagliardo-Nirenberg  inequalities used in the paper (for the original formulation and proof see \cite{Gagliardo}, \cite{Niremberg}, \cite{GagNirProof}). Here, the constant $c$ depends only on $\Om$ and we assume that $f$ is periodic function on $\Om$  it is sufficiently regular to make the right-hand side finite. Firstly, we recall
\eqnsl{
\nck{ \nab f } \le
c  \nd{ \nab f } \nd{ \nabt f }.
}{gag50}
The lower order term (say, $L^{2}$ norm) can be omitted, because $\iOm \nab f dx =0$, $\iOm \nabk f dx =0$ and from Poincar\'e inequality for functions with vanishing mean  we get
\[
\ndk{ \nab f}= \nd{ \nab f} \nd{ \nab f} \leq C_{1} \nd{ \nab f } \nd{ \nabd f}\leq C_{2} \nd{ \nab f } \nd{ \nabt f}  ,
\]
where $C_{1}$, $C_{2}$ depends only on Poincar\'e constant for $\Om$.  Next, we have
\eqnsl{
\nt{f}^{2} \le c \nd{\nab f} \nd{f} , \hd \m{ if } \hd \int_{\Om}fdx=0,
}{est-zero-3}
\eqq{\ns{f} \le c \nd{\nab f}  , \hd \m{ if } \hd \int_{\Om}fdx=0,}{Sobolev}
\eqnsl{
\ns{\nabla f}^{2} \le c \nd{\nabt f} \nd{\nab f},
}{est-nab-6}
\eqnsl{
\nck{\nab f} &
\le c\nd{ \nabt f } \norm{\nab f}_{\frac{3}{2}},
}{est-nab-4-kw}
\eqnsl{
\nif{f}  
\leq c(\nd{ \nabk f} + \nj{f}).
}{gag_inf_lap_j}
\eqnsl{
\nif{f} \le c \nd{ \nabk f}, \hd \m{ if } \hd \int_{\Om}fdx=0,
}{gag_inf_lap}
\eqnsl{
\norm{f}_\td \le c \norm{\nab f}^\jd_\td \nj{f}^\jd + c\nj{f},
}{32Est}
where $c$ depends only on $\Om$.

{\bf Statements and Declarations. } The authors declare that no funds, grants, or other support were received during the preparation of this manuscript. The authors have no relevant financial or non-financial interests to disclose.

\end{document}